\documentclass{amsart}
\usepackage[dvips]{epsfig}
\usepackage{graphics}
\usepackage{latexsym}
\usepackage{amsmath}
\usepackage{amsthm}
\usepackage{amssymb}

\newtheorem{thm}{Theorem}
\newtheorem{assumption}[thm]{Assumption}
\newtheorem{lem}[thm]{Lemma}

\newtheorem{defi}[thm]{Definition}

\newtheorem{prop}[thm]{Proposition}
\newtheorem{rk}[thm]{Remark}

\newenvironment{preuve}{\vip \noindent {\it Proof}}{\hfill$\square$\vip}

\newcommand{\poubelle}[1]{}
\newcommand{\vip}{\vskip.2cm}
\newcommand{\e}{{\varepsilon}}
\newcommand{\dd}{{\mathbb{D}}}
\newcommand{\rr}{{\mathbb{R}}}
\newcommand{\nn}{{\mathbb{N}}}
\newcommand{\GG}{{\mathbb{G}}}
\newcommand{\zz}{{\mathbb{Z}}}
\newcommand{\E}{\mathbb{E}}
\newcommand{\Var}{\mathbb{V}{\rm ar}\,}
\newcommand{\cF}{{\mathcal F}}
\newcommand{\cL}{{\mathcal L}}
\newcommand{\cN}{{\mathcal N}}
\newcommand{\cS}{{\mathcal S}}
\newcommand{\cP}{{\mathcal P}}

\newcommand{\cE}{{\mathcal E}}
\newcommand{\tZ}{{\tilde Z}}
\newcommand{\tH}{{\tilde H}}
\newcommand{\tm}{{\tilde m}}
\newcommand{\tOmega}{{\tilde \Omega}}
\newcommand{\tcF}{{\tilde {\mathcal F}}}
\newcommand{\tPr}{{\tilde \Pr}}

\newcommand{\bM}{{\boldsymbol M}}
\newcommand{\bU}{{\boldsymbol U}}

\newcommand{\bdelta}{{\boldsymbol \delta}}

\newcommand{\bx}{{\boldsymbol x}}

\newcommand{\bm}{{\boldsymbol m}}
\newcommand{\btm}{\tilde{\boldsymbol m}}

\newcommand{\bmu}{{\boldsymbol \mu}}
\newcommand{\indiq}{{{\bf 1}}}
\newcommand{\intot}{{\int_0^t}}
\newcommand{\bW}{{\overline{W}}}
\newcommand{\tW}{{\widetilde W}}
\newcommand{\hW}{{\widehat W}}

\begin{document}

\title{High dimensional Hawkes processes}

\author{Sylvain Delattre, Nicolas Fournier and Marc Hoffmann}

\address{Sylvain Delattre, Laboratoire de Probabilit\'es et Mod\`eles Al\'eatoires, UMR 7599, 
Universit\'e Paris Diderot, Case courrier 7012, avenue de France, 75205 Paris Cedex 13, France.}

\email{sylvain.delattre@univ-paris-diderot.fr}

\address{Nicolas Fournier, Laboratoire de Probabilit\'es et Mod\`eles Al\'eatoires, UMR 7599,
Universit\'e Pierre-et-Marie Curie, Case 188, 4 place Jussieu, F-75252 Paris Cedex 5, France.}

\email{nicolas.fournier@upmc.fr}

\address{Marc Hoffmann, CEREMADE, CNRS-UMR 7534,
Universit\'e Paris Dauphine, Place du mar\'echal De Lattre de Tassigny
75775 Paris Cedex 16, France.}

\email{hoffmann@ceremade.dauphine.fr}

\begin{abstract}
We generalise the construction of multivariate Hawkes processes to a possibly infinite
network of counting processes on a directed graph $\mathbb  G$.
The process is constructed as the solution to a system of Poisson driven stochastic differential equations,
for which we prove pathwise existence and uniqueness under some reasonable conditions.

We next investigate how to approximate a standard $N$-dimensional Hawkes process by a simple 
inhomogeneous Poisson process in the mean-field framework where each pair of 
individuals interact in the same way, in the limit $N \rightarrow \infty$. In the so-called linear case 
for the interaction, we further investigate the large time behaviour of the process.
We study in particular the stability of the central limit theorem when exchanging 
the limits $N, T\rightarrow \infty$ and exhibit different possible behaviours.  

We finally consider the case $\mathbb G = \mathbb Z^d$ with nearest
neighbour interactions. In the linear case, we prove some (large time) laws of large numbers
and exhibit different behaviours, reminiscent of the infinite setting. 
Finally we study the propagation of a {\it single impulsion} started at a given point 
of $\zz^d$ at time $0$.
We compute the probability of extinction of such an impulsion and, in some particular cases, 
we can accurately describe how it 
propagates to the whole space.
\end{abstract}

\maketitle

Limit theorems. Mean-field approximations. Interacting particle systems.

\section{Introduction}

\subsection{Motivation}

In several apparently different applied fields, a growing interest has been observed recently for a better 
understanding of stochastic interactions between multiple entities evolving through time. These include: 
seismology for modelling earthquake replicas (Helmstetter-Sornette \cite{hs}, Kagan \cite{k}, Ogata \cite{o2}, 
Bacry-Muzy \cite{bmu2}), neuroscience for modelling spike trains in brain activity (Gr\"un {\it et al.} 
\cite{gda},  Okatan {\it et al.} \cite{owb}, Pillow  {\it et al.} \cite{psp}, Reynaud  {\it et al.} 
\cite{rrgt, rrt}), genome analysis (Reynaud-Schbath \cite{rs}), financial contagion (Ait-Sahalia  
{\it et al.} \cite{acl}), high-frequency finance (order arrivals, see Bauwens-Hautsch \cite{bh}, Hewlett 
\cite{he}, market micro-structure see Bacry {\it et al.} \cite{bdhm1} and market impact see Bacry-Muzy 
\cite{bmu, bmu2}), financial price modelling across scales (Bacry  {\it et al.} \cite{bdhm2}, 
Jaisson-Rosenbaum \cite{jr}), social networks interactions (Blundell {\it et al.} \cite{bhb}, 
Simma-Jordan \cite{sj}, Zhou  {\it et al.} \cite{zzs}) and epidemiology like for instance viral 
diffusion on a network (Hang-Zha \cite{yz}), to name but a few. 
In all these contexts, observations are often represented as events (like spikes or features) associated 
to agents or nodes on a given network, and  that arrive randomly through time but that are not stochastically 
independent. 

\vip

In practice, we observe a multivariate counting process $(Z_t^1,\ldots, Z_t^N)_{t \geq 0}$, each component $Z_t^i$ 
recording the number of events of the $i$-th component of the system during $[0,t]$, or equivalently  the 
time stamps of the observed events.  Under relatively weak general assumptions, a multivariate counting 
process $(Z_t^1,\ldots, Z_t^N)_{t \geq 0}$ is characterised by its intensity process 
$(\lambda_t^1, \ldots, \lambda_t^N)_{t \geq 0}$, informally defined by
$$\mathrm{Pr}\big(Z^i\;\;\text{has a jump in}\;[t, t+dt]\;\big|\;\mathcal F_t\big) = 
\lambda_t^i dt,\;\;i=1,\ldots, N,$$
where $\mathcal F_t$ denotes the sigma-field generated by 
$(Z^i)_{1 \leq i \leq N}$ up to time $t$.
For modelling the interactions, a particularly attractive family of multivariate point processes 
is given by the class of (mutually exciting) Hawkes processes (Hawkes \cite{h}, Hawkes-Oakes \cite{ho}), with intensity process given by
$$\lambda_t^i=h_i\Big(\sum_{j= 1}^N\int_0^t \varphi_{ji}(t-s)dZ_s^j\Big),$$
where the {\it causal} functions $\varphi_{ji}:[0,\infty)\rightarrow \rr$ model how $Z^j$ acts on $Z^i$ 
by affecting its intensity process $\lambda^i$. The nonnegative functions $h_i$ account for some 
non-linearity, but if we set $h_i(x)=\mu_i+x$ with $\mu_i \geq 0$, we obtain {\it linear Hawkes processes} 
where $\mu_i$ can be interpreted as a baseline Poisson intensity. In the degenerate case $\varphi_{ji}=0$,  
we actually retrieve  standard Poisson processes.

\vip

Multivariate Hawkes processes have long been studied in probability theory (see for instance the comprehensive 
textbook of Daley-Vere-Jones \cite{dvj} and the references therein, Br\'emaud-Massouli\'e \cite{bm} or the 
recent results of Zhu \cite{z1, z2}). Their statistical inference is relatively well understood too, from a 
classical parametric angle (Ogata \cite{o-1}) together with recent significant advances in nonparametrics 
(Reynaud-Bouret-Schbath \cite{rs}, Hansen  {\it et al.} \cite{hrr}). However, the frontier is progressively 
moving to understanding the case of large $N$, when the number of components may become increasingly large 
or possibly infinite (see Galvez-L\"ocherbach \cite{gl} for some constructions in that direction). 
This context is potentially of major importance for future developments in the aforementioned applied 
fields. This is the topic of the present paper. 

\subsection{Setting}

We work on a filtered probability space $(\Omega,\cF,(\cF_t)_{t\geq 0},\Pr)$.
We say that $(X_t)_{t\geq 0}$ is a counting process if it is non-decreasing, c\`adl\`ag, integer-valued (and finite
for all times), with all its jumps of height $1$. For 
$(X_t)_{t\geq 0}$ a $(\cF_t)_{t\geq 0}$-adapted counting process, there is a unique 
non-decreasing predictable process $(\Lambda_t)_{t\geq 0}$, called {\it compensator} of $(X_t)_{t\geq 0}$,
such that $(X_t-\Lambda_t)_{t\geq 0}$ is a $(\cF_t)_{t\geq 0}$-local martingale, 
see Jacod-Shiryaev \cite[Chapter I]{js}.  

\vip

We consider a countable directed graph
$\mathbb G = \big(\mathcal S, \mathcal E\big)$
with vertices (or nodes) $i \in \mathcal S$ and (directed) edges $e \in \mathcal E$.  
We write $e=(j,i) \in \mathcal E$ for the oriented edge. We also need to specify 
the following parameters:
a kernel $\boldsymbol \varphi = (\varphi_{ji}, 
(j,i) \in \mathcal E)$ with  $\varphi_{ji}: [0,\infty) \mapsto \mathbb R$, and a nonlinear intensity component 
$\boldsymbol h = (h_i, i \in \mathcal S)$ with  $h_i:\rr \mapsto [0,\infty)$. The natural generalisation of 
finite-dimensional Hawkes processes is the following.

\begin{defi}\label{df1}
A Hawkes process with parameters 
$(\GG,\boldsymbol \varphi, \boldsymbol h)$ is a family of $(\cF_t)_{t\geq 0}$-adapted counting
processes $(Z^i_t)_{i \in \mathcal S, t \geq 0}$ such that

\vip

(i) almost surely, for all $i \ne j$, $(Z_t^i)_{t \geq 0}$ and $(Z_t^j)_{t \geq 0}$ never jump simultaneously,

\vip

(ii) for every $i \in \mathcal S$, the compensator $(\Lambda_t^i)_{t \geq 0}$ of $(Z_t^i)_{t \geq 0}$ has the form 
$\Lambda_t^i = \int_0^t \lambda^i_s ds$, where  the intensity process $(\lambda_t^i)_{t \geq 0}$ is given by
\begin{equation*}
\lambda^i_t = h_i\Big(\sum_{j \rightarrow i} \int_0^{t}\varphi_{ji}(t-s)d Z_s^j\Big),
\end{equation*}
with the notation $\sum_{j \rightarrow i}$  for summation over $\{j\; : \; (j,i) \in \mathcal E\}$.
\end{defi}

We say that a Hawkes process is {\it linear} 
when $h_i(x)=\mu_i+x$ for every $x\in\rr$, $i \in \mathcal S$, with $\mu_i\geq 0$ and
when $\varphi_{ji}\geq 0$. We will give some general existence, uniqueness and approximation results 
for {\it nonlinear} Hawkes processes, but
all the precise large-time estimates we will prove concern the linear case. 

\vip

A Hawkes process 
$(Z^i_t)_{i \in \mathcal S, t \geq 0}$ with parameters 
$(\GG,\boldsymbol \varphi, \boldsymbol h)$ behaves as follows.
For each $i\in \cS$, the rate of jump of $Z^i$ is, at time $t$, $\lambda_i(t)=h_i(\sum_{j \rightarrow i}
\sum_{k\geq 1} \varphi_{ji}(t-T^j_k)\indiq_{\{T^j_k\leq t\}})$, where $(T^j_k)_{k\geq 1}$ are the jump times
of $Z^j$. In other words, each time one of the $Z^j$'s has a jump, it {\it excites} its {\it neighbours}
in that it increases  their rate of jump (in the natural situation where $h$ is increasing and $\varphi$ is 
positive).
If $\varphi$ is positive and decreases to $0$, the case of almost all applications we have in mind, the 
influence of a jump decreases  and tends to $0$ as time evolves.

\subsection{Main results}

In the case where $\GG$ is a finite graph, under some appropriate assumptions on the parameters, the 
construction of $(Z_t^i)_{i \in \mathcal S, t \geq 0}$ is standard. However, for an infinite graph, the situation
is more delicate: we have to check, in some sense, that the interaction does not come from infinity.

\vip

The first part of this paper (Section \ref{wp}) consists of writing a Hawkes process
as the solution to a system of Poisson-driven S.D.E.s and of
finding a set of assumptions on $\GG$
and on the parameters $(\boldsymbol \varphi, \boldsymbol h)$ under which we can prove
the pathwise existence and uniqueness for this system of S.D.E.s. 
Representing counting processes as solutions to S.D.E.s is classical, 
see Lewis-Shedler \cite{ls}, Ogata, \cite{o}, Br\'emaud-Massouli\'e \cite{bm}, Chevallier \cite{c}.
However, the well-posedness of such S.D.E.s is not obvious when $\GG$ is an infinite graph.

\vip

In a second part (Section \ref{mfl}), we study the {\it mean-field} situation:
we assume that we have a finite (large) number $N$ of {\it particles}
behaving similarly, with no geometry. In other words, $\cS=\{1,\dots,N\}$ is endowed
with the set of all possible edges $\cE=\{(i,j)\;:\; i,j\in\cS\}$,
and there are two functions $h$ and $\varphi$ such that
$h_i=h$ and $\varphi_{ij}=N^{-1}\varphi$ for all $i,j \in \cS$. We show that, as $N\to\infty$,
Hawkes processes can be approximated by an i.i.d.\ family of inhomogeneous Poisson processes.
Concerning the large-time behaviour, we discuss, in the linear case, the possible law of large numbers 
and central limit theorems as $(t,N)\to(\infty,\infty)$ and we observe some different situations 
according to the position of $\int_0^\infty \varphi(t)dt$ with respect to $1$ (the so-called critical case).

\vip

Finally, we consider in Section \ref{nnm} the case where $\GG$ is $\zz^d$,
endowed with the set of edges $\cE=\{(i,j) \; : \; |i-j|=0$ or $1\}$, where $|\cdot|$ denotes the 
Euclidean distance. We study
the large time behaviour, in the linear case where $h_i(x)=\mu_i+x$ and when 
$\varphi_{ij}=(2d+1)^{-1}\varphi$ 
does not depend on $i,j$. We first assume that $\mu_i$ does not depend too much on $i$ (consider {\it e.g.}
the case where the $\mu_i$ are random, i.i.d. and bounded) and show that (i) if $\int_0^\infty \varphi(t)dt>1$, 
then there
is a law of large numbers and the interaction makes everything {\it flat}, in the sense that for all 
$i\ne j$, $Z^i_t\sim Z^j_t$ as $t\to \infty$;
(ii) if $\int_0^\infty \varphi(t)dt<1$, then there is again a law of large numbers, but the limiting value
depends on $i$. We also explain why these results are reminiscent of the infinite setting and of the interaction.
Finally, we study the case where $\mu_i=0$ for all $i$ but where there is an {\it impulsion} at time $0$
at $i=0$. We compute the probability of extinction of such an impulsion and, in some particular
cases, we study how it propagates to the whole space (when it does not blow out).

\subsection{Notation}

The Laplace transform of $\varphi:[0,\infty)\mapsto \rr$ is defined, when it exists, by
$$
\cL_\varphi(\alpha)=\int_0^\infty e^{-\alpha t}\varphi(t) dt.
$$
We also introduce the convolution of $h,g:[0,\infty)\mapsto \rr$ as (if it exists)
$(g\star h)_t=\int_0^t g_sh_{t-s}ds = \int_0^t g_{t-s}h_{s}ds$. As is well-known, when everything makes sense,
$\cL_{g\star h}(\alpha)=\cL_{g}(\alpha)\times\cL_{h}(\alpha)$.

\section{Well-posedness using a Poisson S.D.E.}\label{wp}

We will study Hawkes processes through a system of Poisson-driven stochastic differential
equations. This will allow us to speak of pathwise existence and uniqueness and to
prove some propagation of chaos using some simple {\it coupling} arguments.

\vip

Consider, on a filtered probability space $(\Omega,\cF,(\cF_t)_{t\geq 0}, \Pr)$, a
family $(\pi^i(ds\,dz), i \in \mathcal S)$  of i.i.d. $(\cF_t)_{t\geq 0}$-Poisson
measures with intensity measure $ds\,dz$ on 
$[0,\infty) \times [0,\infty)$.

\begin{defi}\label{df2}
A family $(Z_t^i)_{i\in \cS, t \geq 0}$ of {\it c\`adl\`ag} $(\mathcal F_t)_{t \geq 0}$-adapted 
processes is called a Hawkes process with parameters $(\GG,\boldsymbol \varphi, \boldsymbol h)$ if a.s.,
for all $i\in\cS$, all $t\geq 0$ 
\begin{equation} \label{equation basique}
Z_t^i= \int_0^t \int_0^\infty {\bf 1}_{\displaystyle \big\{z \leq h_i \big(\sum_{j \rightarrow i}\int_0^{s-}\varphi_{ji}(s-u)dZ_u^{j}\big)\big\}}
\pi^i(ds\,dz).
\end{equation}
\end{defi}

This formulation is consistent with Definition \ref{df1}.

\begin{prop}\label{con}

(a) A Hawkes process in the sense of Definition \ref{df2} is also a Hawkes process  
in the sense of Definition \ref{df1}.

\vip

(b) Consider a Hawkes process in the sense of Definition \ref{df1} (on some filtered probability space
$(\Omega,\cF,(\cF_t)_{t\geq 0},\Pr)$. Then we can build, on a
possibly enlarged probability space $(\tOmega,\tcF,(\tcF_t)_{t\geq 0},\tPr)$, a
family $(\pi^i(ds\,dz), i \in \mathcal S)$ of i.i.d. $(\tcF_t)_{t\geq 0}$-Poisson
measures with intensity measure $ds\,dz$ on 
$[0,\infty) \times [0,\infty)$ such that $(Z_t^i)_{i\in \cS, t \geq 0}$ is a Hawkes process in the sense of Definition
\ref{df2}.
\end{prop}

Point (a) is very easy: for  a Hawkes process $(Z_t^i)_{i\in \cS, t \geq 0}$ in the sense of Definition \ref{df2},
it is clear that for every $i\in\cS$, the compensator of $Z^i$ is $\intot \int_0^\infty 
{\bf 1}_{\{z \leq h_i (\sum_{j \rightarrow i}\int_0^{s-}\varphi_{ji}(s-u)dZ_u^{j})\}} dzds$, which is equal to 
$\intot h_i (\sum_{j \rightarrow i}\int_0^{s-}\varphi_{ji}(s-u)dZ_u^{j})ds$. Furthermore, the independence of the
Poisson random measures $(\pi^i(ds\,dz), i \in \mathcal S)$ guarantees that for all $i\ne j$, $(Z^i_t)_{t\geq 0}$
and $(Z^j_t)_{t\geq 0}$ a.s. never jump simultaneously.

\vip

Point (b) is more delicate but standard and a very similar result was given in Br\'emaud-Massouli\'e \cite{bm}.
Their proof is based on results found in the book \cite{j} of Jacod, 
of which one of the main goals is exactly this topic: prove the equivalence
between martingale problems and S.D.E.s. We also refer to Chevallier \cite[Section IV]{c} where a very complete 
proof
is given as well as a historical survey. 
Let us mention that the idea to integrate 
an indicator function with respect
to a Poisson measure in order to produce an inhomogeneous Poisson process with given intensity 
was first introduced by Lewis-Shedler \cite{ls}, and later extended by Ogata \cite{o} in the case of a 
stochastic intensity.

\vip

The following set of assumptions will guarantee the well-posedness of \eqref{equation basique}.

\begin{assumption} \label{basic h, varphi} There are some nonnegative constants $(c_i)_{i\in\cS}$,
some positive weights $(p_i)_{i\in\cS}$ and a locally integrable function $\phi:[0,\infty)
\mapsto [0,\infty)$ such that 

\vip

(a) for every $i\in\cS$, every $x,y\in\rr$, $|h_i(x)-h_i(y)|\leq c_i |x-y|$,

\vip

(b) $\sum_{i \in \mathcal S} h_i(0) p_i<\infty$,

\vip

(c) for every $s\in [0,\infty)$, every $j\in\cS$,
$\sum_{i, (j,i)\in \mathcal E}c_i p_i \big|\varphi_{ji}(s)\big|\leq p_j \phi(s)$.
\end{assumption}

Let us give a few examples of parameters 
$(\GG,\boldsymbol \varphi, \boldsymbol h)$ satisfying Assumption \ref{basic h, varphi}.

\begin{rk}\label{example}
(i) If $\cS$ is finite, then Assumption \ref{basic h, varphi} holds true, with the choice $p_i=1$, 
as soon as $h_i$ is Lipschitz continuous for all $i\in\cS$ and
$\varphi_{ji}$ is locally integrable for all $(j,i)\in\cE$. 

\vip

(ii) If $\cS=\mathbb{Z}^d$ is endowed with
$\cE=\{(i,j)\;:\; |i-j|=0$ or $1\}$, then 
Assumption \ref{basic h, varphi} holds, with the choice $p_i=2^{-|i|}$, if 
$\sum_{i\in\zz^d} 2^{-|i|}|h_i(0)|<\infty$ and if there are $c>0$ and $\varphi \in L^1_{loc}([0,\infty))$
such that $|h_i(x)-h_i(y)|\leq c|x-y|$ and $|\varphi_{jk}(t)|\leq \varphi(t)$
for all $i\in\cS$, $x,y\in\rr$, $(j,k)\in\cE$ and $t\geq 0$.

\vip

(iii) Consider next $\cS=\mathbb{Z}^d$ endowed with the set of all possible edges 
$\cE=\{(i,j)\;:\; i,j \in\zz^d\}$ and assume that there is $c>0$ such that 
$|h_i(0)|\leq c$ and $|h_i(x)-h_i(y)|\leq c|x-y|$ for all $i\in\cS$, $x,y\in\rr$.
Assume that there are $\varphi \in L^1_{loc}([0,\infty))$ and a nonincreasing $a:[0,\infty)\mapsto [0,\infty)$
such that $|\varphi_{ji}(t)|\leq a(|i-j|) \varphi(t)$ for all $(i,j)\in\cE$ and $t\geq 0$. Then
if $\sum_{i\in\zz^d} a(|i|)<\infty$, Assumption \ref{basic h, varphi} holds true.

\vip

(iv) Consider the (strongly oriented) graph $\zz_+$ endowed with the set of edges
$\cE=\{(i,i+1)\;:\; i \in\zz_+\}$. 
Then Assumption \ref{basic h, varphi} holds true as soon as there is 
$\varphi \in L^1_{loc}([0,\infty))$ such that for every $i \in \zz_+$, there are
$c_i>0$ and $a_i>0$ such that $|h_i(x)-h_i(y)|\leq c_i|x-y|$ and $|\varphi_{i(i+1)}|\leq a_i \varphi$.
\end{rk}

Points (ii) and (iii) of course extend to other graphs.
In (iv), there is no growth condition on $|h_i(0)|$, $c_i$ and $a_i$. This comes from the fact that
the interaction is directed: $Z^0$ is actually a Poisson process with rate
$h_0(0)$, the intensity of $Z^1$ is entirely determined by that of $Z^0$, and so on. Hence this example
is not very interesting. But we can mix e.g. points (ii) and (iv): 
informally, coefficients corresponding to edges directed to the
origin have to be well-controlled, while coefficients corresponding to edges directed to infinity
require less assumptions.

\begin{proof} Point (i) is obvious. To check (ii), simply note that for all $j\in\cS$,
$\sum_{i, (j,i)\in \mathcal E} c 2^{-|i|} |\varphi_{ji}|\leq c \varphi 
2^{-|j|} \sum_{i, (j,i)\in \mathcal E} 2^{|j|-|i|} \leq  c 2(2d+1)\varphi 2^{-|j|}$ and define $\phi=c 2(2d+1)\varphi$.
Point (iv) holds with $(p_i)_{i \in \zz_+}$ defined by $p_0=1$ and, by induction, $p_{i+1}= \min\{2^{-i}/(1+h_{i+1}(0)), 
p_i/(1+a_ic_{i+1})\}$. This of course implies that $\sum_{i\in\zz_+}p_i|h_i(0)|<\infty$ and that for all 
$j\geq 1$, $\sum_{i, (j,i)\in \mathcal E} c_i p_i |\varphi_{ji}|=c_{j+1}p_{j+1}|\varphi_{j(j+1)}|
\leq c_{j+1}p_{j+1} a_j \varphi \leq p_j \varphi$ as desired. 

\vip

To prove (iii), we work with the sup norm 
$|i|=|(i_1,\dots,i_d)|=\max\{|i_1|,\dots,|i_d|\}$. The delicate part consists in showing that there is
$b: \nn \mapsto [0,\infty)$ and a constant $C>0$ such that $\sum_{i\in\zz^d}b(|i|)<\infty$
and, for all $j\in\zz^d$,  $\sum_{i\in\zz^d}b(|i|) a(|i-j|) \leq C b(|j|)$. Then the result will easily follow,
with the choices $p_i=b(|i|)$ and $\phi=Cc \varphi$. We define $b$ recursively, by $b(0)=a(0)$
and $b(k+1)=\max\{a(k+1),[(k+1)/(k+2)]^{2d}b(k)\}$. Using that $a$ is nonincreasing, we easily check that
$b$ is nonincreasing. We next check that $\sum_{i\in\zz^d}b(|i|)<\infty$, i.e. that 
$\sum_{k \geq 0}k^{d-1}b(k)<\infty$, knowing by assumption that $\sum_{k \geq 0}k^{d-1}a(k)<\infty$. 
We have, for $k\geq 0$,
\begin{align*}
b(k+1)-a(k+1)=&\big([(k+1)/(k+2)]^{2d}b(k)-a(k+1) \big)_+ \\
\leq& [(k+1)/(k+2)]^{2d} (b(k)-a(k))_+ +\big([(k+1)/(k+2)]^{2d}a(k)-a(k+1) \big)_+\\
\leq & [(k+1)/(k+2)]^{2d}(b(k)-a(k)) + (a(k)-a(k+1)).
\end{align*}
Recalling that $b(0)=a(0)$, one gets 
$b(k)-a(k) \leq \sum_{\ell=1}^k (a(\ell-1)-a(\ell))[(\ell+1)/(k+1)]^{2d}$ by iteration.
Hence 
\begin{align*}
\sum_{k\geq 1} k^{d-1}(b(k)-a(k)) \leq & 
\sum_{k\geq 1} k^{d-1}\sum_{\ell=1}^k (a(\ell-1)-a(\ell))[(\ell+1)/(k+1)]^{2d}\\
=&\sum_{\ell \geq 1}  (a(\ell-1)-a(\ell))(\ell+1)^{2d} \sum_{k \geq  \ell} k^{d-1}(k+1)^{-2d} \\
\leq&C \sum_{\ell \geq 1}  (a(\ell-1)-a(\ell))\ell^{d}.
\end{align*}
This last quantity is nothing but $C \sum_{\ell \geq 1}  a(\ell)[(\ell+1)^d-\ell^d]\leq C  
\sum_{\ell \geq 1} a(\ell) \ell^{d-1}<\infty$. We have thus checked that 
$\sum_{k \geq 0}k^{d-1}b(k)=\sum_{k \geq 1} a(k)k^{d-1}+ \sum_{k\geq 1} k^{d-1}(b(k)-a(k))<\infty$.

\vip

We finally prove that for all $j\in\zz^d$,  $\sum_{i\in\zz^d}b(|i|) a(|i-j|) \leq C b(|j|)$. 
First, we claim that there is $C$ such that $b(k)\leq C b(2k)$ for all $k\geq 0$.
This is easily checked, iterating the inequality $b(k)\leq [(k+2)/(k+1)]^{2d}b(k+1)$. Next we write,
using that $a$ and $b$ are nonincreasing,
\begin{align*}
\sum_{i\in\zz^d}b(|i|) a(|i-j|)\leq & \sum_{|i|<|j|/2}b(|i|) a(|i-j|)+\sum_{|i|\geq |j|/2}b(|i|) a(|i-j|) \\
\leq & a(|j|/2) \sum_{i \in \zz^d}b(|i|) + b(|j|/2)\sum_{i \in \zz^d} a(|i-j|) \\
\leq & Ca(|j|/2)+ C b(|j|/2).
\end{align*}
By definition of $b$, we have $a(|j|/2) \leq b(|j|/2)$. And we have just seen that $b(|j|/2) \leq Cb(|j|)$.
We finally have checked that  $\sum_{i\in\zz^d}b(|i|) a(|i-j|) \leq C b(|j|)$ as desired.
\end{proof}

Our well-posedness result is the following.

\begin{thm} \label{existence uniqueness Y}
Under Assumption \ref{basic h, varphi}, there exists a pathwise unique 
Hawkes process  $(Z_t^i)_{i\in\cS,t \geq 0}$ such that
$\sum_{i \in \mathcal S} p_i \E[Z_t^i]<\infty$ for all $t\geq 0$.
\end{thm}

Observe that this result is not completely obvious in the case of an infinite graph.
In some sense, we have to check that the interaction does not come from infinity.
Let us insist on the fact that, even in simple situations, a {\it graphical construction}
is not possible: consider {\it e.g.} the case of $\zz$ endowed with the set of edges $\cE=
\{(i,j) \; : \; |i-j|=0$ or $1\}$,
assume that $h_i(x)=1+x$ for all $i\in\cS$ and that $\varphi_{ij}=1$ for all $(i,j)\in\cE$.
Then one easily gets convinced that
we cannot determine the values of $(Z^0_t)_{t\in [0,T]}$ by observing the Poisson
measures $\pi^i$ in a (random) finite box.

\vip

As a second comment, let us mention that we believe it is not possible, or at least quite difficult, to 
obtain the full uniqueness, {\it i.e.} uniqueness outside the class of processes satisfying 
$\sum_{i \in \mathcal S} p_i \E[Z_t^i]<\infty$ (or something similar). Indeed, consider again the case of $\zz$ 
endowed with $\cE=\{(i,j) \; : \; |i-j|=0$ or $1\}$,
assume that $h_i(x)=1+x$ for all $i\in\cS$ and that $\varphi_{ji}=1$ for all $(i,j)\in\cE$.
One easily checks that for $(Z_t^i)_{i\in\cS,t \geq 0}$ a Hawkes process, for $m^i_t=\E[Z^i_t]$, it holds
that $m^i_t= t + \intot (m^{i-1}_s+m^i_s+m^{i+1}_s)ds$ for every $i$. This infinite system of equations
is of course closely related to the heat equation $\partial_t u(t,x)=1 + \partial_{xx}u(t,x)$ 
on $[0,\infty)\times \rr$ and with initial condition $u(0,x)=0$. As is well-known, uniqueness for
this equation fails to hold true without imposing some growth conditions as $|x|\to \infty$.
See e.g. Tychonov's counterexample of uniqueness, which can be found in John \cite[Chapter 7]{jo}.

\begin{proof} We first prove uniqueness. Let thus $(Z_t^i)_{i\in\cS,t \geq 0}$ and $(\tZ_t^i)_{i\in\cS,t \geq 0}$
be two solutions to \eqref{equation basique} satisfying the required condition. Set
$$
\Delta_t^i=\int_0^t \big|d\big(Z_s^i-\tZ_s^i\big)\big|\;\;\text{for}\;\;i\in 
\mathcal S,t\geq 0.
$$
In other words, $\Delta^i_t$ is the total variation norm of the signed measure 
$d\big(Z_s^i-\tZ_s^i\big)$ on $[0,t]$. We also put $\delta_t^i = \E[\Delta_t^i]$ and first prove that
\begin{equation} \label{Tassoeur}
\delta_t^i \leq  c_i \int_0^t \sum_{j \rightarrow i}|\varphi_{ji}(t-s)| \delta_s^jds. 
\end{equation}
We have
\begin{align*}
\Delta_t^i  = & \int_0^t \int_0^\infty \Big| 
{\bf 1}_{\big\{z \leq h_i \big(\sum_{j \rightarrow i}\int_0^{s-}\varphi_{ji}(s-u)dZ_u^{j}\big)\big\}} 
-{\bf 1}_{\big\{z \leq h_i \big(\sum_{j \rightarrow i}\int_0^{s-}\varphi_{ji}(s-u)d\tZ_u^{j}\big)\big\}}
\Big|  \pi^i(ds\,dz). 
\end{align*}
Taking expectations, we deduce that
\begin{align}
\delta_t^i  = & \int_0^t \mathbb E\Big[\Big|h_i \big(\sum_{j \rightarrow i}\int_0^{s-}
\varphi_{ji}(s-u)dZ_u^{j}\big)-h_i \big(\sum_{j \rightarrow i}\int_0^{s-}\varphi_{ji}(s-u)
d\tZ_u^{j}\big)\Big|\Big]ds \nonumber\\
\leq & c_i \sum_{j \rightarrow i}\mathbb E\Big[\int_0^t \int_0^{s-}\big|\varphi_{ji}(s-u)\big|
d\Delta_u^{j}ds\Big]\label{Tassmere}
\end{align}
by Assumption \ref{basic h, varphi}-(a). Using Lemma \ref{tlt}, we see that
\begin{align*}
\int_0^t ds \int_0^{s-}|\varphi_{ji}(s-u)|d\Delta_u^j =
\int_0^{t}\big|\varphi_{ji}(t-u)\big|\Delta_u^j du
\end{align*}
which, plugged into \eqref{Tassmere}, yields \eqref{Tassoeur}. 

\vip

Set $\delta_t = \sum_{i \in \mathcal S}p_i\delta_t^i$, where 
the weights $p_i$ were introduced in Assumption \ref{basic h, varphi}.  By assumption, $\delta_t$ 
is well-defined and finite. We infer by \eqref{Tassoeur} that
\begin{align*}
\delta_t  \leq \int_0^t \sum_{i \in \mathcal S} p_i c_i \sum_{j \rightarrow i}
\big|\varphi_{ji}(t-s)\big|\delta_s^j  ds.
\end{align*}
By Assumption \ref{basic h, varphi}-(c),
\begin{align*}
\delta_t \leq & \int_0^t \sum_{j \in \mathcal S} \delta_s^j 
\sum_{i, (j,i)\in \mathcal E}c_i p_i\big|\varphi_{ji}(t-s)\big|ds 
\leq \int_0^t \sum_{j \in \mathcal S}p_j\delta_s^j\phi(t-s)ds= \int_0^t \phi(t-s)\delta_sds.
\end{align*}
Lemma \ref{grrrr}-(i) thus implies that $\delta_t=0$ identically, from which uniqueness follows.

\vip

We now quickly prove existence by a Picard iteration. 
Let $Z^{i,0}_t = 0$ and, for $n\geq 0$,
\begin{equation} \label{pic}
Z_t^{i,n+1}= \int_0^t \int_0^\infty 
{\bf 1}_{\{z \leq h_i (\sum_{j \rightarrow i}\int_0^{s-}\varphi_{ji}(s-u)dZ_u^{j,n})\}}
\pi^i(ds\,dz).
\end{equation}
We define $\delta^{i,n}_t=\E[\intot |dZ_s^{i,n+1}-dZ_s^{i,n}|]$
and  $\delta^{n}_t= \sum_{i\in \cS} p_i \delta^{i,n}_t $.
As in the proof of uniqueness, we obtain, for $n\geq 0$,
\begin{align}\label{tass1}
\delta_t^{n+1} \leq &  \int_0^t \phi(t-s)\delta_s^{n}ds.
\end{align}
Next, we put $m^{i,n}_t=\E[Z_t^{i,n}]$.
By Assumption \ref{basic h, varphi}-(a), $h_i(x)\leq h_i(0)+c_i|x|$, whence
\begin{align*}
m^{i,n+1}_t \leq & \E\Big[\intot \Big(h_i(0)+c_i\sum_{j \rightarrow i}\int_0^{s-}
|\varphi_{ji}(s-u)|dZ^{j,n}_u \Big) ds   \Big]
\leq  \intot \Big(h_i(0) + c_i\sum_{j \rightarrow i} |\varphi_{ji}(t-s)|m^{j,n}_s \Big) ds,
\end{align*}
where we used that, by Lemma \ref{tlt},
$\intot \int_0^{s-} |\varphi_{ji}(s-u)|dZ^{j,n}_u ds =\intot |\varphi_{ji}(t-u)| Z^{j,n}_u du $.
Setting  $u_t^n=\sum_{i\in\cS} p_i m^{i,n}_t$ and using Assumption \ref{basic h, varphi}-(b)-(c),
\begin{align}\label{tass2}
u^{n+1}_t \leq& 
t \sum_{i\in \cS} h_i(0) p_i 
+ \intot \sum_{i\in \cS} p_i c_i \sum_{j \rightarrow i} |\varphi_{ji}(s-u)|m^{j,n}_s  ds
\leq  C t 
+ \intot \phi(t-s) u^n_s ds.
\end{align}
Since $u^0_t=0$ and $\phi$ is locally integrable, 
we easily check by induction that $u^n$ is locally bounded for all $n\geq 0$.
Consequently, $\delta^n$ is also locally bounded for all $n\geq 0$. Lemma \ref{grrrr}-(ii)
implies that for all $T\geq 0$, $\sum_{n\geq 1}\delta_T^n<\infty$. 
This classically implies that the Picard sequence is Cauchy and thus converges:
there exists a family $(Z_t^i)_{i\in\cS,t \geq 0}$ of  {\it c\`adl\`ag} 
nonnegative adapted processes such that for all $T\geq 0$,
$\lim_n \sum_{i\in\cS} p_i \E[\int_0^T |dZ_s^{i}-dZ_s^{i,n}| ]=0$.
It is then not hard to pass to the limit in \eqref{pic} to deduce that $(Z_t^i)_{i\in\cS,t \geq 0}$
solves \eqref{equation basique}.
Finally, Lemma \ref{grrrr}-(iii) implies that $\sup_n u^n_t <\infty$ for all $t\geq 0$, from which 
$\sum_{i \in \mathcal S} p_i \mathbb E[Z_t^i]<\infty$ as desired.
\end{proof}

\section{Mean-field limit}\label{mfl}

In this section, we work in the following setting.

\begin{assumption} \label{mean field hyp}
Let $h:\rr\mapsto [0,\infty)$ be such that
$|h|_{lip}=\sup_{x\ne y}|x-y|^{-1}|h(x)-h(y)|<\infty$
and let $\varphi=[0,\infty)\mapsto \rr$ be a locally square integrable function.

For each $N\geq 1$, we consider the complete graph
$\mathbb G_N$ with vertices $\mathcal S_N = \{1,\ldots, N\}$ and edges 
$\mathcal E_N = \{(i,j) \; : \; i,j\in \mathcal S_N\}$, {\it i.e.} all pairs of points in $\mathcal S_N$ 
are connected. We put $h_i^N=h$ for all $i\in\cS_N$ and $\varphi_{ji}^N=N^{-1}\varphi$
for $(i,j)\in\cE_N$.
\end{assumption}

Under Assumption \ref{mean field hyp}, the triplet $(\GG_N,\boldsymbol \varphi^N,\boldsymbol h^N)$
satisfies Assumption \ref{basic h, varphi} (the graph $\mathbb G_ N$ is finite) for each $N\geq 1$.
Therefore, a Hawkes process $(Z_t^{N,1},\ldots, Z_t^{N,N})_{t \geq 0}$ with parameters 
$(\GG_N,\boldsymbol \varphi^N,\boldsymbol h^N)$
is uniquely defined by Theorem \ref{existence uniqueness Y}.

\vip

Introduce the {\it limit} equation
\begin{equation} \label{sol mean-field}
\overline Z_t = \int_0^t\int_0^\infty {\bf 1}_{\displaystyle \big\{z \leq h\big(\int_0^s 
\varphi(s-u)d\E[\overline Z_u]\big)\big\}}\pi(ds\,dz),\;\;\text{for every}\;\;t\geq 0,
\end{equation}
where $\pi(ds \, dz)$ is a Poisson measure on $[0,\infty)\times [0,\infty)$ with intensity measure $ds dz$.
Also, $d\E[\overline Z_u]$ is the measure on $[0,\infty)$ associated to the (necessarily) 
non-decreasing function  $u\mapsto \E[\overline Z_u]$.
Note that a solution 
$\overline Z = ({\overline Z}_t)_{t \geq 0}$, if it exists, is an inhomogeneous Poisson process 
on $[0,\infty)$ with intensity $\lambda_t=h(\int_0^t\varphi(t-u)d\E[\overline Z_u])$.

\subsection{Propagation of chaos}

The main result of this section reads as follows.

\begin{thm} \label{mean-field approx}
Work under Assumption \ref{mean field hyp}.

\vip

(i) There is a pathwise unique solution $({\overline Z}_t)_{t\geq 0}$ to \eqref{sol mean-field} 
such that $(\E[{\overline Z}_t])_{t\geq 0}$ is locally bounded.

\vip

(ii) It is possible to build simultaneously the Hawkes process $(Z^{N,1}_t,\dots,Z^{N,N}_t)_{t\geq 0}$ 
with parameters $(\GG_N,\boldsymbol \varphi^N,\boldsymbol h^N)$
and an i.i.d. family $({\overline Z}^i_t)_{t\geq 0,i=1,\dots,N}$ of solutions to \eqref{sol mean-field}
in such a way that for all $T>0$, all $i=1,\dots,N$,
$$
\E\Big[\sup_{[0,T]} |Z^{N,i}_t -{\overline Z}^i_t |\Big] \leq C_T N^{-1/2},
$$
the constant $C_T$ depending only on $h$, $\varphi$ and $T$ (see Remark \ref{vitesse} below
for some bounds of $C_T$ in a few situations).

\vip

(iii) Consequently, we have the mean-field approximation
$$\frac{1}{N}\sum_{i = 1}^N \delta_{(Z^{N,i}_t)_{t\geq 0}}\longrightarrow 
{\mathcal L}\big(({\overline Z}_t)_{t\geq 0}\big)\;\;
\text{in probability, as}\;\;N\rightarrow \infty,
$$
where $\cP(\dd([0,\infty),\rr))$ is endowed with the weak convergence topology associated
with the topology (on $\dd([0,\infty),\rr)$) of the uniform convergence on compact time intervals.
\end{thm}

\begin{proof}
For $({\overline Z}_t)_{t\geq 0}$ a solution to \eqref{sol mean-field}, the equation satisfied by 
$m_t=\E[{\overline Z}_t]$ writes
\begin{equation} \label{sol moyenne}
m_t = \int_0^t h\big( \int_0^s \varphi(s-u)dm_u\big)ds\;\;\text{for every}\;\;t\geq 0.
\end{equation}
By Lemma \ref{expuni}, 
we know that this equation has a unique non-decreasing locally bounded solution, which furthermore
is of class $C^1$ on $[0,\infty)$.
We now split the proof in several steps.

\vip

{\it Step 1.} Here we prove the well-posedness of \eqref{sol mean-field}. 
For $({\overline Z}_t)_{t \geq 0}$ a solution to \eqref{sol mean-field},
its expectation $m_t=\E[{\overline Z}_t]$ solves \eqref{sol moyenne} and is thus uniquely defined.
Thus the right hand side of \eqref{sol mean-field} is uniquely
determined, which proves uniqueness. For the existence, consider $m$ the unique
solution to \eqref{sol moyenne} and put ${\overline Z}_t=
\int_0^t\int_0^\infty {\bf 1}_{\{z \leq h(\int_0^s \varphi(s-u)dm_u)\}}\pi(ds\,dz)$.
We thus only have to prove that $\E[{\overline Z}_t]=m_t$. But 
$\E[{\overline Z}_t]=\intot h(\int_0^s \varphi(s-u)dm_u)ds$, which is nothing but $m_t$
since $m$ solves \eqref{sol moyenne}.

\vip

{\it Step 2.} We next introduce a suitable coupling.
Let $(\pi^i(ds\,dz))_{i\geq 1}$ be an i.i.d. family of Poisson
measures with common intensity measure $dsdz$ on
$[0,\infty) \times [0,\infty)$. For each $N\geq 1$, we consider the Hawkes process
$(Z^{N,1}_t,\dots,Z^{N,N}_t)_{t\geq 0}$
\begin{align*}
Z^{N,i}_t=&\int_0^t \int_0^\infty {\bf 1}_{\big\{z \leq h \big(N^{-1}\sum_{j=1}^N\int_0^{s-}\varphi(s-u)dZ_u^{N,j}\big)\big\}}
\pi^i(ds\,dz).
\end{align*}
Next, still denoting by $m$ the unique solution to \eqref{sol moyenne}, we put, for
every $i\geq 1$,
\begin{align*}
\overline{Z}^{i}_t=&\int_0^t \int_0^\infty {\bf 1}_{\big\{z \leq h \big( \int_0^{s-}\varphi(s-u)dm_u\big)\big\}}
\pi^i(ds\,dz).
\end{align*}
Clearly, $((\overline{Z}^{i}_t)_{t\geq 0})_{i\geq 1}$ is an i.i.d. family of solutions to
\eqref{sol mean-field}.

\vip

{\it Step 3.} Here we introduce $\Delta_{N}^{i}(t)= \intot |d({\overline Z}^i_u-Z^{N,i}_u)|$
and  $\delta_N(t)=\E[\Delta_{N}^{i}(t)]$, which obviously does not depend on $i$ (by exchangeability).
Observe that 
\begin{align}\label{tictac}
\sup_{[0,t]} |{\overline Z}^i_u-Z^{N,i}_u| \leq \Delta_N^i(t), \quad \hbox{whence} \quad
\E\Big[\sup_{[0,t]} |{\overline Z}^i_u-Z^{N,i}_u| \Big] \leq \delta_N(t).
\end{align}
We show in this step that for all $t>0$,
\begin{align}\label{tactic}
\delta_N(t)\leq |h|_{lip} N^{-1/2} 
\intot \Big( \int_0^s \varphi^2(s-u) dm_u \Big)^{1/2}ds
+ |h|_{lip}\intot |\varphi(t-s)| \delta_N(s) ds.
\end{align}
First, $\Delta_{N}^{1}(t)$ equals
$$
\int_0^t \int_0^\infty
\Big|{\bf 1}_{\big\{z \leq h \big(N^{-1}\sum_{j=1}^N\int_0^{s-}\varphi(s-u)dZ_u^{N,j}\big)\big\}}
-{\bf 1}_{\big\{z \leq h \big(\int_0^{s-}\varphi(s-u)dm_u\big)\big\}}\Big|
\pi^i(ds\,dz).
$$
Taking expectations, we find
\begin{align*}
\delta_N(t)=&\intot \E\Big[\Big|h \big( \int_0^{s}\varphi(s-u)dm_u\big)
-h \big(N^{-1}\sum_{j=1}^N\int_0^{s}\varphi(s-u)dZ_u^{N,j}\big)\Big| \Big] ds,
\end{align*}
whence
\begin{align}
\delta_N(t)\leq & |h|_{lip} \intot \E\Big[\Big|\int_0^{s}\varphi(s-u)dm_u
-N^{-1}\sum_{j=1}^N\int_0^{s}\varphi(s-u)d{\overline Z}_u^{j}\Big| \Big] ds \nonumber\\
&+ |h|_{lip} 
\intot \E\Big[\Big|N^{-1}\sum_{j=1}^N\int_0^{s}\varphi(s-u)d[{\overline Z}_u^{j} -Z_u^{N,j}]\Big| 
\Big] ds  \nonumber \\
=& |h|_{lip}(A+ B). \label{of1}
\end{align}
Using exchangeability and Lemma \ref{tlt},
\begin{align}\label{of1bis}
B \leq& \intot \E\Big[\int_0^{s}|\varphi(s-u)|d\Delta_{N}^{1}(u)\Big] ds
= \intot |\varphi(t-u)| \delta_N(u) du.
\end{align}
Next, we use that $X^j_s=\int_0^{s}\varphi(s-u)d{\overline Z}_u^j$ are i.i.d. with mean
$\int_0^{s}\varphi(s-u)dm_u$, whence 
\begin{align}\label{of2}
A \leq N^{-1/2} \intot (\Var X^1_s)^{1/2} ds.
\end{align}
But it holds that
$$
X^1_s=\int_0^s \int_0^\infty \indiq_{\{z\leq h(\int_0^{u}\varphi(u-r)dm_r)\}} \varphi(s-u) \pi^1(du\,dz).
$$
Since the integrand is deterministic, denoting by $\tilde \pi^1$ the compensated Poisson measure,
$$
X^1_s-\E[X^1_s]=\int_0^s \int_0^\infty \indiq_{\{z\leq h(\int_0^{u}\varphi(u-r)dm_r)\}}
\varphi(s-u) \tilde \pi^1(du\,dz).
$$
Recalling Assumption \ref{mean field hyp}, we find
\begin{align}\label{of3}
\Var X^1_s = & \int_0^s \varphi^2(s-u) h\big(\int_0^u \varphi(u-r)dm_r\big)ds
= \int_0^s \varphi^2(s-u) dm_u.
\end{align}
We used \eqref{sol moyenne} for the last inequality. Gathering \eqref{of1}, \eqref{of1bis}, \eqref{of2}
and \eqref{of3} completes the step.

\vip

{\it Step 4.} Here we conclude that for all $T\geq 0$,
$\sup_{[0,T]}\delta_N(t) \leq  C_T N^{-1/2}$. This will end the proof of (ii) by \eqref{tictac}.
This is not hard: it suffices to start from \eqref{tactic}, to apply Lemma \ref{grrrr}-(i) and to observe that
$\intot \Big( \int_0^s \varphi^2(s-u) dm_u \Big)^{1/2}ds$ is locally bounded 
(which follows from the assumption that $\varphi$ is locally square integrable
and the fact that $m$ is $C^1$ on $[0,\infty)$).

\vip

{\it Step 5.} Finally, (iii) follows from (ii): by Sznitman \cite[Proposition 2.2]{s},
it suffices to check that for each fixed $\ell\geq 1$, $((Z^{N,1}_t)_{t\geq 0},\dots,(Z^{N,\ell}_t)_{t\geq 0})$
goes in law, as $N\to\infty$, to $\ell$ independent copies of $(\overline{Z}_t)_{t\geq 0}$
(for the uniform topology on compact time intervals). This clearly follows from (ii).
\end{proof}

We now want to show that the constant $C_T$ we get can be quite satisfactory.

\begin{rk}\label{vitesse} Work under Assumption \ref{mean field hyp}.

\vip

(a) Assume that $|h|_{lip}\int_0^\infty |\varphi(s)|ds <1$ (subcritical case) and that 
$\int_0^\infty \varphi^2(s)ds<\infty$. 
Then (ii) of Theorem \ref{mean-field approx} holds with
$C_T=CT$, for some constant $C>0$. This is a satisfactory slow growth.

\vip

(b) Assume that $h(x)=\mu+x$ for some $\mu>0$ and that 
$\varphi(t)=a e^{-b t}$ for some $a>b>0$ (if $a<b$, then point (a) applies).
Then $m_t=\E[\overline{Z}_t]\sim \mu a (a-b)^{-2} e^{(a-b)t}$ as $t\to \infty$
and (ii) of Theorem \ref{mean-field approx} holds with
$C_T=C e^{(a-b)T}$, for some constant $C>0$. This is again quite satisfactory: the error is of order
$N^{-1/2}m_T$.
\end{rk}

\begin{proof}
We start with (a). Using the notation of the previous proof, it suffices (see \eqref{tictac}) to show that
$\delta_N(T) \leq C T N^{-1/2}$.
Setting 
$\Lambda = |h|_{lip}\int_0^\infty |\varphi(s)|ds <1$, starting from \eqref{tactic} 
and observing that $\delta_N$ is non-decreasing, 
we find
$\delta_N(t) \leq |h|_{lip}N^{-1/2}\intot (\int_0^s \varphi^2(s-u)dm_u)^{1/2}ds 
+ \Lambda \delta_N(t)$,
whence $\delta_N(t) \leq CN^{-1/2} \intot (\int_0^s \varphi^2(s-u)dm_u)^{1/2}ds$. 
We thus only have to check that $\int_0^s \varphi^2(s-u)dm_u$ is bounded on $[0,\infty)$. 
Since $\int_0^\infty \varphi^2(s)ds<\infty$, it suffices to prove that $m'$ is bounded on $[0,\infty)$.
But $m'_t = h(\int_0^t \varphi(t-u)m'_udu) \leq h(0)+  |h|_{lip} \int_0^t |\varphi(t-u)|m'_udu$, whence
$\sup_{[0,T]} m'_t \leq   h(0)+ \Lambda \sup_{[0,T]} m'_t$ and thus $\sup_{[0,T]} m'_t \leq   h(0)/(1-\Lambda)$ 
for any $T>0$.

\vip

We next check (b). First, \eqref{sol moyenne} rewrites 
$m_t=\mu t +  a \intot \int_0^s e^{-b(s-u)}dm_u ds$, with unique solution 
$$
m_t=\frac{-\mu b t}{ a -b}  + \frac{\mu a(e^{( a -b)t} -1)}
{( a -b)^2} \sim \frac{\mu a}
{( a -b)^2}  e^{( a -b)t}.
$$
Next, using \eqref{tactic} and the explicit
expressions of $h$, $\varphi$ and $m$, we find 
\begin{align*}
\delta_N(t) \leq&  N^{-1/2}\intot \Big(\int_0^s \varphi^2(s-u)dm_u\Big)^{1/2}ds
+ \intot \varphi(t-s)\delta_N(s)ds \\
\leq & C N^{-1/2} e^{( a -b)t /2} + a \intot e^{-b(t-s)}\delta_N(s)ds.
\end{align*}
Setting $u_N(t)=\delta_N(t)e^{b t}$, we get $u_N(t)\leq
C N^{-1/2} e^{( a + b)t /2} + a\intot u_N(s)ds$. By  Gr\"onwall's lemma,
$u_N(t) \leq C N^{-1/2} e^{( a + b)t /2} + 
a \intot C N^{-1/2} e^{( a + b)s /2}e^{a(t-s)}ds$. 
On easily deduces, since $ a>b$, that 
$u_N(t) \leq C N^{-1/2} e^{ a t}$ so that
$\delta_N(t) \leq C N^{-1/2} e^{(a-b)t}$. The use of \eqref{tictac} ends the proof.
\end{proof}

\subsection{Large time behaviour}

We now address the important problem of the large time behaviour.
Since the solution $(\overline{Z}_t)_{t\geq 0}$ to \eqref{sol mean-field} is nothing but
an inhomogeneous Poisson process, its large-time behaviour is easily and precisely described, provided we have
sufficiently information on the solution to \eqref{sol moyenne}.
The question is thus: can we use the large time estimates of the mean-field limit to describe the
large-time behaviour of the true Hawkes process with a large number of particles?
To fix the ideas, we consider the linear case. 
\vip

We treat separately the subcritical and supercritical cases.

\begin{thm}\label{PlimdlimT}
Work under Assumption \ref{mean field hyp} with $\varphi$ nonnegative and $h(x)=\mu+x$ for some $\mu>0$. 
Assume also that $\Lambda=\int_0^\infty \varphi(s) ds <1$.
For each $N\geq 1$, consider the Hawkes process $(Z^{N,1}_t,\dots,Z^{N,N}_t)_{t\geq 0}$
with parameters $(\GG_N,\boldsymbol \varphi^N,\boldsymbol h^N)$. Consider also
the unique solution $(m_t)_{t\geq 0}$ to \eqref{sol moyenne}.

\vip

1. We have $m_t \sim a_0 t$ as $t \to \infty$, where $a_0=\mu/(1-\Lambda)$.

\vip

2. For any fixed $i\geq 1$, $Z^{N,i}_t/m_t$ tends to $1$ in probability as $t\to \infty$, uniformly in $N$.
More precisely, $\E[|Z^{N,i}_t/m_t -1|] \leq C m_t^{-1/2}$ for some constant $C$.

\vip

3. For any fixed $\ell\geq 1$
$(m_t^{1/2}(Z^{N,i}_t/m_t-1))_{i=1,\dots,\ell}$ goes in law to $\cN(0,I_\ell)$ as $(t,N)\to (\infty,\infty)$ 
(without condition on the regime).
\end{thm}

\begin{thm}\label{GlimdlimT}
Work under Assumption \ref{mean field hyp} with $\varphi$ nonnegative and $h(x)=\mu+x$ for some $\mu>0$.
Assume also that  $\Lambda=\int_0^\infty \varphi(s) ds \in (1,\infty]$. 
Assume finally that $t\mapsto \intot |d\varphi(s)|$  has at most polynomial growth. 
For each $N\geq 1$, consider the Hawkes process $(Z^{N,1}_t,\dots,Z^{N,N}_t)_{t\geq 0}$
with parameters $(\GG_N,\boldsymbol \varphi^N,\boldsymbol h^N)$.
Consider also the unique solution $(m_t)_{t\geq 0}$ to \eqref{sol moyenne}.

\vip

1. We  have $m_t \sim a_0 e^{\alpha_0 t}$ 
as $t \to \infty$, where $\alpha_0>0$ is determined by $\cL_\varphi(\alpha_0)=1$
and where $a_0=\mu \alpha_0^{-2}(\int_0^\infty t \varphi(t)e^{-\alpha_0 t}dt)^{-1}$.

\vip

2. For any fixed $i\geq 1$,
$Z^{N,i}_t/m_t$ tends to $1$ in probability as $(t,N)\to (\infty,\infty)$.
More precisely, there is a constant $C$ such that 
$\E[|Z^{N,i}_t/m_t -1|] \leq m_t^{-1/2} + CN^{-1/2}(1+m_t^{-1})$.

\vip

3. For any fixed $\ell\geq 1$,

\vip

(i)  $(m_t^{1/2}(Z^{N,i}_t/m_t-1))_{i=1,\dots,\ell}$ goes in law to $\cN(0,I_\ell)$ if
$t\to \infty$ and $N\to\infty$ with $m_t/N \to 0$;

\vip

(ii)  $(N^{1/2}(Z^{N,i}_t/m_t-1))_{i=1,\dots,\ell}$ goes in law to $(X,\dots,X)$,
if $t\to \infty$ and $N\to\infty$ with $m_t/N \to \infty$. Here $X$ is a $\cN(0,\sigma^2)$-distributed 
random variable, where $\sigma^2=\alpha_0^{2}\mu^{-2} \int_0^\infty e^{-2\alpha_0s}m'_sds$.
\end{thm}

Let us summarize. At first order (law of large numbers), the mean-field approximation is always
good for large times. At second order (central limit theorem), the mean field approximation is
always good for large times in the subcritical case, but fails to be relevant for too large times (depending
on $N$) in the supercritical case: the independence property breaks down.

\vip

In the supercritical case, we have the technical condition that $t\mapsto \intot |d\varphi(s)|$  has at most 
polynomial growth. This is useful to have some precise estimates of the solution $m$ to \eqref{sol moyenne}.
This is, {\it e.g.} always satisfied when $\varphi$ is bounded and non-increasing, as is
often the case in applications. It is slightly restrictive however, since it forces $\varphi(0)$ to be finite.

\vip

It should be possible to study also the {\it critical} case, but then the situation is more intricate:
many regimes might arise.
With a little more work, we could also study, in the supercritical case, the regime where 
$m_t/N \to x\in(0,\infty)$.

\vip

In order to prove Theorems \ref{PlimdlimT} and \ref{GlimdlimT}, we will use the following central limit 
theorem for martingales.

\begin{lem}\label{jactass}
Let $\ell\geq 1$ be fixed. For $N\geq 1$, consider a family $(M^{N,1}_t,\dots,M^{N,\ell}_t)_{t\geq 0}$ of 
$\ell$-dimensional local martingales satisfying $M^{N,i}_0=0$.
Assume that all their jumps are uniformly bounded and that $[M^{N,i},M^{N,j}]_t=0$
for every $N\geq 1$, $i\ne j$ and $t\geq 0$. Assume also that there is a continuous increasing 
function $(v_{t})_{t\geq 0}:[0,\infty)\mapsto[0,\infty)$ such that for all $i=1,\dots,\ell$,
$\lim_{(t,N)\to(\infty,\infty)} v^{-2}_{t}[M^{N,i},M^{N,i}]_t = 1$ in probability. In the case where 
$v_\infty=\lim_{t\to \infty} v_t<\infty$,
assume moreover that for all $i=1,\dots,\ell$, all $t_0>0$, uniformly in $t\geq t_0$,
$\lim_{N\to\infty} [M^{N,i},M^{N,i}]_t = v^{2}_{t}$ in probability.

Then
$v^{-1}_{t}(M^{N,1}_t,...,M^{N,\ell}_t)$ converges in law to the Gaussian distribution $\cN(0,I_\ell)$
as $(t,N)\to(\infty,\infty)$, where $I_\ell$ is the $\ell\times\ell$ identity matrix.
\end{lem}

\begin{proof}
Let $(t_N)_{N\ge 1}$ be a sequence of positive numbers such that $t_N \to\infty$.
We want to prove that $v^{-1}_{t_N}(M^{N,1}_{t_N},...,M^{N,\ell}_{t_N})$ converges in law to $\cN(0,I_\ell)$.
For all $u\in[0,1]$, set
$$\tau^N_u=\inf\{t\ge 0 : v^2_{t} \ge u\,v^{2}_{t_N}\}.$$
Since $v$ is increasing and continuous, $\tau^N$ is also continuous and increasing for each $N$.
We also clearly have $v^2_{\tau^N_u} = u v^{2}_{t_N}$ for all $u\in [0,1]$ and $\tau^N_1=t_N$.
Finally, for each $u>0$ fixed, the sequence $\tau^N_u$ is increasing.
\vip

For all $u\in(0,1]$, $\lim_N v^{-2}_{\tau^N_u}[M^{N,i},M^{N,i}]_{\tau^N_u}=1$ in probability. 
Indeed, in the case $v_\infty=\infty$, this follows from the facts that $\lim_N \tau^N_u=\infty$
and $\lim_{(t,N)\to(\infty,\infty)} v^{-2}_{t}[M^{N,i},M^{N,i}]_t = 1$.
When $v_\infty<\infty$, the additional assumption (uniformity in $t\geq t_0$ of the convergence as $N\to\infty$)
clearly suffices, since the sequence $\tau^N_u$ is increasing and thus bounded from below.

\vip

We define the martingales $(L^{N,i}_u)_{u\in[0,1]}$ by $L^{N,i}_u=v_{t_N}^{-1}M^{N,i}_{\tau^N_u}$.
All their jumps are uniformly bounded (because those of $M^{N,i}$ are assumed to be uniformly bounded and
because $\sup_N v_{t_N}^{-1}<\infty$ since $v$ is increasing).
We also have $[L^{N,i},L^{N,j}]_u =0$ for all $i\not=j$, all $u\in [0,1]$.
Furthermore, using that $v^2_{\tau^N_u}=u v^{2}_{t_N}$,
\begin{align*}
[L^{N,i},L^{N,i}]_u &= \frac{[M^{N,i},M^{N,i}]_{\tau^N_u}}{v^2_{t_N}}=\frac{[M^{N,i},M^{N,i}]_{\tau^N_u}}{v^2_{\tau^N_u}}\,u
\to \,u
\end{align*}
in probability.
Therefore, according to Jacod-Shiryaev \cite{js} (Theorem VIII-3.11), the process 
$(L^{N,1}_u,\dots,L^{N,\ell}_u)_{u\in [0,1]}$ 
converges in law to $(B^1_u,\dots,B^\ell_u)_{u\in[0,1]}$ where the $B^i$ are independent standard Brownian motions.
In particular, $(L^{N,1}_1,\dots,L^{N,\ell}_1)$ goes in law to $\cN(0,I_\ell)$.
To conclude the proof, it thus suffices to observe that $L^{N,i}_1= v_{t_N}^{-1}M^{N,i}_{\tau^N_1}=
v_{t_N}^{-1}M^{N,i}_{t_N}$.
\end{proof} 

We can now give the

\begin{preuve} {\it of Theorem \ref{PlimdlimT}.} 
In the present (linear) case, we can rewrite \eqref{sol moyenne} as 
$m_t=\intot (\mu + \int_0^s \varphi(s-u)dm_u)ds=\mu t + \intot \varphi(t-s)m_sds$ by Lemma \ref{tlt}.
This equation is studied in details in Lemma \ref{Plemlim}: recalling that 
$\Lambda=\int_0^\infty \varphi(s)ds<1$, we have $m'_t \sim a_0$ and $m_t \sim a_0 t$ as $t\to \infty$,
where $a_0=\mu/(1-\Lambda)$, which proves point (i).
The proof is now divided in several steps. Step 1 will also be used in the supercritical
case.

\vip

{\it Step 1.} 
Recall that, for some i.i.d. family $(\pi^i(ds \, dz))_{i\geq 1}$ of Poisson measures on 
$[0,\infty)\times[0,\infty)$ with intensity measure $dsdz$,
\begin{align*}
Z^{N,i}_t=\intot\int_0^\infty \indiq_{\big\{z \leq \mu + N^{-1}\sum_{j=1}^N\int_0^{s-}\varphi(s-u)dZ_u^{N,j}\big\}}\pi^i(ds\,dz).
\end{align*}
We have $\E[Z^{N,i}_t]=m_t$. Indeed, by exchangeability, we see that 
$\E[Z^{N,i}_t]=\E[Z^{N,1}_t]$ and that
$$
\E[Z^{N,1}_t]=\intot \Big(\mu+ N^{-1}\sum_{j=1}^N\int_0^s \varphi(s-u) d \E[Z^{N,j}_u] \Big)ds
=\intot \Big(\mu+ \int_0^s \varphi(s-u) d \E[Z^{N,1}_u] \Big)ds,
$$
whence $(\E[Z^{N,1}_t])_{t\geq 0}$ solves \eqref{sol moyenne}, of which the unique solution is $(m_t)_{t\geq 0}$
by Lemma \ref{expuni}.

\vip

We next introduce $U^{N,i}_t=Z^{N,i}_t-m_t$ and the martingales (here $\tilde \pi^i(ds \, dz)= \pi^i(ds \, dz)-dsdz$)
\begin{align*}
M^{N,i}_t=\intot\int_0^\infty \indiq_{\big\{z \leq \mu + N^{-1}\sum_{j=1}^N\int_0^{s-}\varphi(s-u)dZ_u^{N,j}\big\}}\tilde \pi^i(ds\,dz).
\end{align*}
We consider the mean processes $\overline{Z}^N_t=N^{-1}\sum_1^N Z^{N,i}_t$, $\overline{U}^N_t=N^{-1}\sum_1^N U^{N,i}_t$
and finally $\overline{M}^N_t=N^{-1}\sum_1^N M^{N,i}_t$. An easy computation using \eqref{sol moyenne} 
and Lemma \ref{tlt} shows that
\begin{align*}
U^{N,i}_t=&M^{N,i}_t+ \intot N^{-1}\sum_{j=1}^N\int_0^{s}\varphi(s-u)dZ_u^{N,j} ds - m_t\\
=&M^{N,i}_t+ \intot \int_0^{s}\varphi(s-u)\Big( N^{-1}\sum_{j=1}^N dZ_u^{N,j} - dm_u \Big) ds\\ 
=&M^{N,i}_t+ \intot \varphi(t-s) \Big( N^{-1}\sum_{j=1}^N Z_s^{N,j} - m_s \Big) ds,
\end{align*}
so that
\begin{align}\label{bol}
U^{N,i}_t= M^{N,i}_t+\intot \varphi(t-s) \overline{U}^N_s ds.
\end{align}
This directly implies that
\begin{align}\label{dair}
\overline{U}^N_t=\overline{M}^N_t+ \intot \varphi(t-s) \overline{U}^N_s ds.
\end{align}
Next, we observe that $[M^{N,i},M^{N,j}]_t=0$ for all $i\ne j$ (because these martingales a.s. never jump
simultaneously) and that $[M^{N,i},M^{N,i}]_t=Z^{N,i}_t$. Hence $[\overline{M}^N,\overline{M}^N]_t=N^{-1} 
\overline{Z}^N_t$. We thus have $\E[(M^{N,i}_t)^2]=\E[Z^{N,i}_t]=m_t$ and 
$\E[(\overline{M}^N_t)^2]=N^{-1}\E[\overline{Z}^N_t]=N^{-1}m_t$.

\vip

{\it Step 2.} Recalling \eqref{dair} and using that $\Lambda=\int_0^\infty \varphi(s)ds<1$,
we observe that $\sup_{[0,t]}|\overline{U}^N_s| \leq \sup_{[0,t]}|\overline{M}^N_s| + 
\Lambda \sup_{[0,t]}|\overline{U}^N_s|$. Consequently,
$$
\E\Big[\sup_{[0,t]}|\overline{U}^N_s| \Big] \leq (1-\Lambda)^{-1} \E\Big[\sup_{[0,t]}|\overline{M}^N_s|\Big ] 
\leq C N^{-1/2}m_t^{1/2}
$$
by the Doob and Cauchy-Schwarz inequalities. We easily deduce that
$$
\E\Big[\intot \varphi(t-s) |\overline{U}^N_s| ds\Big]\leq \Lambda \E\Big[\sup_{[0,t]}|\overline{U}^N_s| \Big] 
\leq C N^{-1/2}m_t^{1/2},
$$
whence finally, recalling \eqref{bol},
$$
m_t^{-1}\E[|U^{N,i}_t|] \leq m_t^{-1}\E[|M^{N,i}_t|] + Cm_t^{-1} N^{-1/2}m_t^{1/2} \leq Cm_t^{-1/2}.
$$
This says that $\E[|Z^{N,i}_t/m_t-1|]\leq C m_t^{-1/2}$ and thus proves point 2.

\vip

{\it Step 3.} We then fix $\ell\geq 1$ and use \eqref{bol} to write, for $i=1,\dots,\ell$,
$$
m_t^{1/2} (Z^{N,i}_t/m_t-1) = m_t^{-1/2}U^{N,i}_t=m_t^{-1/2} M^{N,i}_t
+ m_t^{-1/2} \intot \varphi(t-s) \overline{U}^N_s ds.
$$
First, $\E[m_t^{-1/2}\intot \varphi(t-s) |\overline{U}^N_s| ds]\leq C N^{-1/2}$, which tends to $0$ 
as $(t,N)\to(\infty,\infty)$, by the estimate proved in Step 2. To conclude the proof of point 3, we thus
only have to prove that $(m_t^{-1/2}M^{N,i}_t)_{i=1,\dots,\ell}$
goes in law to $\cN(0,I_\ell)$ as $(t,N)\to(\infty,\infty)$.
To this end, we apply Lemma \ref{jactass}.  The jumps of the martingales $M^{N,i}$ are uniformly bounded 
(by $1$) and we have seen that $[M^{N,i},M^{N,j}]_t=0$ for all $i\ne j$. The function $(m_t)_{t\geq 0}$ is continuous
and increases to infinity.
It thus suffices to check that, as $(t,N)\to (\infty,\infty)$, 
$m_t^{-1}[M^{N,i},M^{N,i}]_t\to 1$ in probability.
Since $[M^{N,i},M^{N,i}]_t=Z^{N,i}_t$, this is an immediate consequence of point 2.
\end{preuve}

We now turn to the supercritical case.

\begin{preuve} {\it of Theorem \ref{GlimdlimT}.} We rewrite \eqref{sol moyenne} as 
$m_t=\intot (\mu + \int_0^s \varphi(s-u)dm_u)ds=\mu t + \intot \varphi(t-s)m_sds$ by Lemma \ref{tlt}.
This equation is studied in details in Lemma \ref{Glemlim}: there is a unique $\alpha_0>0$ such that
$\cL_\varphi(\alpha_0)=1$ and, defining $a_0 \in(0,\infty)$ as in the statement, we have
$m_t \sim a_0 e^{\alpha_0 t}$ and $m'_t \sim a_0 \alpha_0 e^{\alpha_0 t}$ as $t\to \infty$, which proves point 1. 
We also know that
$\Gamma(t)=\sum_{n\geq 1} \varphi^{\star n}(t) \sim (a_0 \alpha_0^2/\mu) e^{\alpha_0 t}$, 
that $\Upsilon(t)=\intot \Gamma(s)ds \sim (a_0 \alpha_0/\mu) e^{\alpha_0 t}$ and further properties
of $m,m',\Gamma,\Upsilon$ are proved in Lemma \ref{Glemlim}.

\vip

{\it Step 1.} We adopt the same notation as in the proof of Theorem \ref{PlimdlimT}-Step 1, of which all 
the results remain valid in the present case. Point 1 follows from Lemma \ref{Glemlim}-(a).

\vip

{\it Step 2.} First, \eqref{dair} says exactly that 
$\overline{U}^N = \overline{M}^N +\varphi\star \overline{U}^N$. Using Lemma \ref{Glemlim}-(e), we deduce that 
$\overline{U}^N = \overline{M}^N + \Gamma \star \overline{M}^N$
(the processes $\overline{U}^N$ and $\overline{M}^N$
are clearly a.s. c\`adl\`ag and thus locally bounded).
Since $\E[(\overline{M}^N_t)^2]=N^{-1}m_t$ by Step 1,
\begin{align*}
\E[|\overline{U}^N_t|] \leq& \E[|\overline{M}^N_t|] + \intot \Gamma(t-s)\E[|\overline{M}^N_s|]ds\\
\leq& N^{-1/2}m_t^{1/2} + \intot \Gamma(t-s) N^{-1/2}m_s^{1/2}ds \\ 
\leq& C N^{-1/2}(1+m_t).
\end{align*}
The last inequality
easily follows from Lemma \ref{Glemlim}-(b).

\vip

Using \eqref{dair} again, we see that $\intot \varphi(t-s)\overline{U}^N_sds= \overline{U}^N_t - \overline{M}^N_t$,
whence
$$
\E\Big[\Big|\intot \varphi(t-s)\overline{U}^N_s \Big|\Big] \leq \E[|\overline{M}^N_t|] + \E[|\overline{U}^N_t|]
\leq N^{-1/2}m_t^{1/2}+ C N^{-1/2}(1+m_t)\leq C N^{-1/2}(1+m_t).
$$
On the other hand, we know from Step 1 that $\E[(M^{N,i}_t)^2]=m_t$. Using \eqref{bol}, we conclude that
$$
\E[|Z^{N,i}_t/m_t - 1|] 
=m_t^{-1}\E[|U^{N,i}_t|] \leq m_t^{-1/2} + CN^{-1/2}(1+m_t^{-1}).
$$
which ends the proof of 2. 

\vip

{\it Step 3.} We then fix $\ell\geq 1$ and write, for $i=1,\dots,\ell$, by \eqref{bol},
$$
(Z^{N,i}_t/m_t-1) = m_t^{-1}U^{N,i}_t=m_t^{-1}M^{N,i}_t + m_t^{-1}\intot \varphi(t-s) \overline{U}^N_s ds.
$$

{\it Step 3.1.} We first consider the regime $(t,N)\to(\infty,\infty)$ with $m_t/N\to 0$ and study
$$
m_t^{1/2}(Z^{N,i}_t/m_t-1) = m_t^{-1/2}M^{N,i}_t + m_t^{-1/2}\intot \varphi(t-s) \overline{U}^N_s ds.
$$
The second term tends to $0$ in probability, because we can bound, using Step 2, 
its $L^1$-norm by  $C m_t^{-1/2}N^{-1/2}(1+m_t)$, which tends to $0$ in the present regime. 
We thus just have to prove that $(m_t^{-1/2}M^{N,i}_t)_{i=1,\dots,\ell}$
goes in law to $\cN(0,I_\ell)$. We use Lemma \ref{jactass}:
the martingales $M^{N,i}$ have uniformly bounded (by $1$) jumps and
we have seen that $[M^{N,i},M^{N,j}]_t=0$ for $i\ne j$. The function $(m_t)_{t\geq 0}$ is continuous
and increases to infinity.
It only remains to check that $m_t^{-1}[M^{N,i},M^{N,i}]_t$ tends to $1$ in probability. But 
$[M^{N,i},M^{N,i}]_t=Z^{N,i}_t$, so that the conclusion follows from point 2.

\vip

{\it Step 3.2.} We finally consider the regime $(t,N)\to(\infty,\infty)$ with $m_t/N\to \infty$
and study
$$
N^{1/2}(Z^{N,i}_t/m_t-1) = N^{1/2}m_t^{-1}M^{N,i}_t + N^{1/2}m_t^{-1}\intot \varphi(t-s) \overline{U}^N_s ds.
$$
First, $N^{1/2}m_t^{-1}M^{N,i}_t \to 0$ in probability, because its $L^1$-norm is bounded by $N^{1/2}m_t^{-1/2}$
(recall that $\E[(M^{N,i}_t)^2]=m_t$),  which tends to $0$ in the present regime.
Since $V_t^N:=N^{1/2}m_t^{-1} \intot \varphi(t-s) \overline{U}^N_s ds$ does not depend on $i$, it only
remains to prove that $V_t^N$ goes in law to $\cN(0,\sigma^2)$.
We write, using \eqref{dair}, recalling that $\overline{U}^N = \overline{M}^N + \Gamma \star \overline{M}^N$
(see Step 2) 
and integrating by parts (recall that $\Upsilon(t)=\int_0^t \Gamma(s) ds$)
$$
V^N_t = N^{1/2}m_t^{-1} (\overline{U}^N_t - \overline{M}^N_t)=  N^{1/2}m_t^{-1} \intot \Gamma(t-s) \overline{M}^N_s ds
=  N^{1/2}m_t^{-1} \intot \Upsilon(t-s) d\overline{M}^N_s.
$$
Introduce $W_t^N=(\alpha_0/\mu) N^{1/2} \intot e^{-\alpha_0 s} d\overline{M}^N_s$ and observe that, since
$\E[[\overline{M}^N,\overline{M}^N]_t]=N^{-1}m_t$,
\begin{align*}
\E[(V^N_t-W^N_t)^2]=&  \E\Big[ N \intot (m_t^{-1}\Upsilon(t-s) - (\alpha_0/\mu) e^{-\alpha_0 s})^2 
d[\overline{M}^N,\overline{M}^N]_s \Big]\\
=& \intot (m_t^{-1}\Upsilon(t-s) - (\alpha_0/\mu) e^{-\alpha_0 s})^2 m'_sds.
\end{align*}
Lemma \ref{Glemlim}-(c) tells us that this tends to $0$ as $t\to \infty$.
We thus only have to prove that 
$W_t^N$ goes in law to $\cN(0,\sigma^2)$ as $(t,N)\to(\infty,\infty)$.

\vip

This follows again from Lemma \ref{jactass} (with $\ell=1$): 
the jumps of the martingale $(W_t^N)_{t\geq 0}$ are bounded by $(\alpha_0/\mu)N^{-1/2}$ (because those of  
$\overline{M}^N$ are bounded by $N^{-1}$). 
The function $v_t=(\alpha_0/\mu)(\intot e^{-2\alpha_0 s} m'_s ds)^{1/2}$ 
is continuous and increasing to the finite limit 
$v_\infty=\sigma$ (which was defined in the statement).
We thus only have to prove that (a) $v_t^{-2}[W^N,W^N]_t \to 1$ in probability as $(t,N)\to(\infty,\infty)$,
(b) for all $t_0>0$, uniformly in $t\geq t_0$, 
$v_t^{-2}[W^N,W^N]_t \to 1$ in probability as $N\to\infty$.
By Lemma \ref{jactass}, we will deduce that $v_t^{-1} W_t^N$ goes in law to $\cN(0,1)$ as  
$(t,N)\to(\infty,\infty)$, which of course implies that  $W_t^N$ goes in law to $\cN(0,\sigma^2)$ as desired.

\vip

We have, since $[\overline{M}^N,\overline{M}^N]_t=N^{-1} \overline{Z}^N_t$,
$$
[W^N,W^N]_t= (\alpha_0/\mu)^2 N \intot e^{-2\alpha_0 s} 
d[\overline{M}^N,\overline{M}^N]_s=(\alpha_0/\mu)^2 \intot e^{-2\alpha_0 s} d \overline{Z}^N_s.
$$
Using that $\overline{Z}^N_t=\overline{U}^N_t+m_t$ and performing an integration by parts, we see that
\begin{align*}
[W^N,W^N]_t=& v_t^2 + (\alpha_0/\mu)^2 \intot e^{-2\alpha_0 s} d \overline{U}^N_s \\
=& v_t^2 + (\alpha_0/\mu)^2 e^{-2\alpha_0 t}\overline{U}^N_t 
+2(\alpha_0^3/\mu^2)\intot e^{-2\alpha_0 s} \overline{U}^N_s ds.
\end{align*}
Recalling that that $\E[|\overline{U}^N_t|] \leq C N^{-1/2}(1+m_t)$, we infer
\begin{align*}
&\E\Big[\Big|(\alpha_0/\mu)^2 e^{-2\alpha_0 t}\overline{U}^N_t 
+2(\alpha_0^3/\mu^2)\intot e^{-2\alpha_0 s} \overline{U}^N_s ds \Big|\Big]\\
 \leq &
\frac C  {N^{1/2}} \Big(e^{-2\alpha_0 t}(1+m_t) +  \intot e^{-2\alpha_0 s} (1+m_s)ds \Big),
\end{align*}
which is bounded by $C N^{-1/2}$ by Lemma \ref{Glemlim}. We have proved that
$\sup_{t\geq 0} \E[ |[W^N,W^N]_t - v_t^2|] \leq C N^{-1/2}$, from which 
points (a) and (b) above immediately follow. The proof is complete.
\end{preuve}

\section{Nearest neighbour model}\label{nnm}

We consider here the case where $\mathbb G$ is a regular grid, on which particles interact
(directly) only if they are neighbours. We will work on $\zz^d$,
endowed with the set of edges 
$$
\cE=\{(i,j) \in (\zz^d)^2 \; :\; |i-j|=0 \hbox{ or } 1\},
$$
where $|(i_1,\dots,i_d)|=(\sum_{r=1}^di_r^2)^{1/2}$. Thus each point has $2d+1$ neighbours (including itself).
We hesitated to include self-interaction, but this avoids some needless complications
due to the periodicity of the underlying random walk on $\zz^d$.

\begin{assumption}\label{tic}
(i) The graph $\mathbb G=(\cS,\cE)$ is $\cS=\zz^d$ (for some $d\geq 1$)
endowed with the above set of edges $\cE$.

(ii) There is a nonnegative 
locally integrable function $\varphi:[0,\infty)\mapsto [0,\infty)$
such that for all $(j,i)\in \cE$, $\varphi_{ji}=(2d+1)^{-1}\varphi$.

(iii) For all $i\in\zz^d$, there is $\mu_i\geq 0$ such that $h_i(x)=\mu_i+x$.
The family $(\mu_i)_{i\in\zz^d}$ is bounded.
\end{assumption}

We next introduce a few notation. 
In the whole section, we call {\it vector} (and write in bold)
a family of numbers indexed by $\zz^d$. We call {\it matrix} a family indexed by $\zz^d\times\zz^d$.
The identity matrix $I$ is of course defined as $I(i,j)=\indiq_{\{i=j\}}$.
We will often use the product of a matrix and a vector.
The matrix $A=(A(i,j))_{i,j\in\zz^d}$ defined by
\begin{align}\label{dfA}
A(i,j)=(2d+1)^{-1}\indiq_{\{(i,j)\in\cE\}}
\end{align}
will play an important role. Since $A$ is a stochastic matrix, we can define,
for any $\Lambda \in (0,1)$,
\begin{align}\label{dfQ}
Q_\Lambda(i,j)=\sum_{n \geq 0} \Lambda^n A^n(i,j).
\end{align}

\subsection{Large-time behaviour}

Under Assumption \ref{tic}, we can use Theorem
\ref{existence uniqueness Y} (with $p_i=2^{-|i|}$, see Remark \ref{example}-(ii)): 
there is a unique Hawkes process
$(Z^i_t)_{i\in\zz^d,t\geq 0}$ with parameters $(\GG,\boldsymbol \varphi, \boldsymbol h)$
such that $\sum_{i\in\zz^d} 2^{-|i|}\E[Z^i_t]<\infty$.
Let us state the first results of this section. As usual,
we treat separately the subcritical and supercritical cases.

\begin{thm}\label{nn1}
Work under Assumption \ref{tic} and assume further that 
$\Lambda=\int_0^\infty \varphi(t)dt<1$. 
Consider the unique Hawkes process $(Z^i_t)_{i \in \zz^d,t\geq 0}$
with parameters $(\GG,\boldsymbol \varphi, \boldsymbol h)$.
For all $i\in\zz^d$, 
$t^{-1}Z^i_t$ goes in probability, as $t\to \infty$, to $\sum_{j \in \zz^d}Q_\Lambda(i,j)\mu_j$.
\end{thm}

\begin{thm}\label{nn2}
Work under Assumption \ref{tic} and assume further that 
$\Lambda=\int_0^\infty \varphi(t)dt \in (1,\infty]$ and that $t\mapsto \intot |d\varphi(s)|$
has at most polynomial growth. Consider $\alpha_0>0$ uniquely defined by $\cL_\varphi(\alpha_0)=1$.
Assume finally that the ``mean value''
\begin{align}\label{condimu}
\mu = \lim_{r\to\infty} \frac1{\#\{i\in\zz^d\; : \; |i|\leq r\}}   \sum_{|i|\leq r}\mu_i
\quad \hbox{ exists and is positive.}
\end{align}
Consider the unique Hawkes process $(Z^i_t)_{i \in \zz^d, t\geq 0}$
with parameters $(\GG,\boldsymbol \varphi, \boldsymbol h)$.
Then for all $i\in\zz^d$, $e^{-\alpha_0 t}Z^i_t$ goes in probability, as $t\to \infty$,
to $a_0=\mu \alpha_0^{-2}(\int_0^\infty t \varphi(t)e^{-\alpha_0 t}dt)^{-1}$.
\end{thm}

Let us comment on these results. In the subcritical case, the parameter $\bmu=(\mu_i)_{i \in\zz^d}$ is strongly 
present in the
limiting behaviour: the limit of $t^{-1}Z^i_t$ depends on a certain mean of $\bmu$ around the site $i$
and thus depends on $i$.
In the supercritical case, the behaviour is very different: the limit value of $e^{-\alpha_0 t}Z^i_t$
does not depend on $i$, and depends on $\bmu=(\mu_i)_{i \in\zz^d}$ only through a {\it global} mean value.
Observe also that for a finite-dimensional (e.g. {\it scalar}) Hawkes process, there is no law of large numbers:
one can get a limit of something like $e^{-\alpha_0 t}Z^i_t$, but the limit is random, see Zhu 
\cite[Section 5.4]{z3} (in particular Theorem 23 and Corollary 1). 
In that sense, we can say that in the supercritical case, 
the law of large number is reminiscent of the infinite dimension and of the interaction.

\vip

We will need a precise approximation for $A^n(i,j)$ where $A$ is defined by \eqref{dfA}. It is given by the 
local central limit theorem,  since $A$ is the transition matrix of an aperiodic symmetric random walk on 
$\zz^d$ with bounded jumps. Precisely, we infer from Lawler-Limic \cite[Theorem 2.1.1 and (2.5)]{ll} that 
there is a constant $C$ such that for
all $n\geq 1$, all $i \in \zz^d$,
\begin{align}\label{ltcl}
|A^n(0,i)-p_n(i)| \leq \frac C {n^{(d+2)/2}}
\end{align}
where, for $t>0$ and $x\in\rr^d$,
\begin{equation} \label{noyauGauss}
p_t(x)= \Big(\frac{2d+1}{4\pi t}\Big)^{d/2}\exp\Big(-\frac{(2d+1)|x|^2}{4 t}\Big).
\end{equation}
To apply \cite[Theorem 2.1.1]{ll}, we needed to compute the covariance matrix $\Gamma$ 
corresponding to our random walk, we found $\Gamma=2.(2d+1)^{-1} I_d$, $I_d$ being the $d\times d$
identity matrix.

\begin{lem}\label{mass}
Consider the matrix $(A(i,j))_{i,j\in\zz^d}$ defined by \eqref{dfA}.

\vip

(i) It holds that $\e_n=\sum_{j\in\zz^d} (A^n(i,j))^2$ does not depend on $i\in\zz^d$ and tends to $0$
as $n\to \infty$.

\vip

(ii) Let $\bmu=(\mu_i)_{i\in\zz^d}$ be bounded and satisfy \eqref{condimu}.
Then for all $i\in \zz^d$, $\lim_{n\to\infty} (A^n \bmu)_i =\mu$.
\end{lem}

\begin{proof} 
In the following we denote by $C$ a constant depending only on $d$.

\vip

Point (i) is easy: since $A^n(i,j)=A^n(0,j-i)$ and since $A$ is stochastic, one has 
$\e_n=\sum_{j\in\zz^d} (A^n(0,j))^2 \le \sup_{j\in\zz^d} A^n(0,j)$. 
Moreover, by \eqref{ltcl}, $A^n(0,j)\le p_n(j)+ C n^{-(d+2)/2} \le C n^{-d/2}$. We conclude
that $\e_n \leq C n^{-d/2}\to 0$ as desired.

\vip

Now we turn to the proof of (ii). Let $i\in\zz^d$.
First we show that $\lim_n [(A^n \bmu)_i-\sum_{j\in\zz^d} p_n(j) \mu_j ]=0$.
Since $(A^n \bmu)_i=\sum_{j\in\zz^d} A^n(i,j)\mu_j =\sum_{j\in\zz^d} A^n(0, j-i)\mu_j$
and since the family $(\mu_j)_{j\in\zz^d}$ is bounded, it suffices to prove that
$v_n =\sum_{j\in \zz^d} | A^n(0,j-i)-p_n(j)|\to 0$. We write
$$ 
v_n \le \sum_{|j|\le n^{1/2+1/4d} } | A^n(0,j-i)-p_n(j) |
+ \sum_{|j|> n^{1/2+1/4d} } ( A^n(0,j-i)+p_n(j) )=v_n^1+v_n^2.
$$
On the one hand, using that $\sum_{j\in\zz^d} |j|^2 A^n(0,j) \le C n$ (the variance of the random walk at time $n$ is 
of order $n$), so that $\sum_{j\in\zz^d} |j|^2 A^n(0,j-i) = \sum_{k\in\zz^d} |i+k|^2 A^n(0,k) \le C (|i|^2+n)$ and
thus
$$ 
\sum_{|j|>n^{1/2+1/4d}} A^n(0,j-i) \le C {n^{-1-1/2d}} \sum_{j\in\zz^d} |j|^2 A^n(0,j-i) \le C n^{-1-1/2d}(|i|^2+n).
$$
Similarly, we have $\sum_{j\in\zz^d} |j|^2 p_n(j)\le C n$ and thus
$$ 
\sum_{|j|>n^{1/2+1/4d}} p_n(j) \le  {n^{-1-1/2d}} \sum_{j\in\zz^d} |j|^2 p_n(j)\le C n^{-1/2d}.
$$
Consequently $\lim_n v^2_n=0$. On the other hand,
$$
v^1_n\le \sum_{|j|\le n^{1/2+1/4d}}\! | A^n(0,j-i)-p_n(j-i) | + \sum_{|j|\le n^{1/2+1/4d}}\! | p_n(j-i)-p_n(j) |.
$$
From \eqref{ltcl}, the first sum is bounded by 
$C n^{-(d+2)/2} \#\{j\in\zz^d : |j|\le n^{1/2+1/4d} \}\leq C n^{-3/4} \to 0$. For the second sum, we use that, 
with $c_d=(2d+1)/4$,
$$ 
| p_n(j-i)-p_n(j) | = p_n(j) \Bigl| 1- \exp\bigl(-\frac{c_d}{n} |i|^2 
+ \frac{2c_d}{n} i.j\bigr)\Bigr|.
$$
Hence for $|j|\le n^{1/2+1/4d}$ and for $n$ large enough (e.g. so that $|i|n^{-1/2+1/4d} \leq 1$),
$$
| p_n(j-i)-p_n(j) | \le C p_n(j) (|i|^2 n^{-1}+ |i| n^{-1/2+1/4d}) \leq C p_n(j) (1+|i|^2) n^{-1/4}.
$$
Thus $\sum_{|j|\le n^{1/2+1/2d}} | p_n(j-i)-p_n(j) | \le C(1+|i|^2) n^{-1/4}$ and we deduce that  $\lim_n v^1_n=0$.

\vip

We have shown that $\lim v_n=0$. It only remains to check that
$\lim_n \sum_{j\in\zz^d} \mu_j p_n(j) = \mu$. Let $(r_k)_{k\ge 0}$ be the increasing sequence of nonnegative numbers 
such that $\{r_k\}_{k\ge 0}=\{|j| \; : \; j\in\zz^d\}$ and observe that
$$
\sum_{j\in\zz^d} \mu_j p_n(j)=\sum_{k\geq 0} p_n(r_k) \sum_{|j|=r_k} \mu_j.
$$
A discrete integration by parts shows that 
$$
\sum_{j\in\zz^d} \mu_j p_n(j)=\sum_{k\geq 0} (p_n(r_k)-p_n(r_{k+1})) \sum_{|j|\leq r_k} \mu_j
=\sum_{k\geq 0} v(r_k)(p_n(r_k)-p_n(r_{k+1})) \frac{1}{v(r_k)}\sum_{|j|\leq r_k} \mu_j,
$$
where $v(r)= \#\{j\in\zz^d\; : \; |j| \le r\}$. 
We easily conclude that $\lim_n \sum_{j\in\zz^d} \mu_j p_n(j) = \mu$ as desired, because

\vip

(a) $\lim_{k\to\infty} \frac{1}{v(r_k)} \sum_{|j|\le r_k} \mu_j = \mu$;

(b) for all $k\ge 0$ fixed, $\lim_n v(r_k)(p_n(r_k)-p_n(r_{k+1}))=0$;

(c) $\lim_{n \to \infty} \sum_{k =0}^{\infty} v(r_k)(p_n(r_k)-p_n(r_{k+1}))=1$.

\vip

Point (a) follows from our condition \eqref{condimu} on $\bmu$, point (b) is obvious (because 
$|v(r_k)(p_n(r_k)-p_n(r_{k+1}))| \leq  v(r_k) \sup_{i\in\zz^d} p_n(i)\leq Cv(r_k) n^{-d/2}\to 0$). To check (c), we
write  $\sum_{k =0}^{\infty} v(r_k)(p_n(r_k)-p_n(r_{k+1}))= \sum_{j \in \zz^d} p_n(j)
=\sum_{j\in\zz^d} A^n(0,j)+  \sum_{j \in \zz^d} [p_n(j)-A^n(0,j)]=1+\sum_{j \in \zz^d} [p_n(j)-A^n(0,j)]$.
This tends to $1$, because $\lim_n \sum_{j \in \zz^d} |p_n(j)-A^n(0,j)|=0$, as seen in the first part of the proof
(this is $v_n$ in the special case where $i=0$).

\end{proof}

Let us now give the

\begin{preuve} {\it of Theorem \ref{nn1}.}
We split the proof into several steps. We assume that there is at least one $i\in\zz^d$ such that
$\mu_i>0$, because else the result is obvious (because then $Z^i_t=0$ for all $i\in \zz^d$, all $t\geq 0$).
The first step will also be used in the supercritical case.

\vip

{\it Step 1.} We write as usual, for some i.i.d. family
$(\pi^i(ds \, dz))_{i\geq 1}$ of Poisson measures on $[0,\infty)\times[0,\infty)$ with intensity measure
$dsdz$,
\begin{align*}
Z^{i}_t=\intot\int_0^\infty \indiq_{\big\{z \leq \mu_i + (2d+1)^{-1}\sum_{j\rightarrow i} \int_0^{s-}\varphi(s-u)dZ_u^{j}\big\}}
\pi^i(ds\,dz).
\end{align*}
Let us put $m^i_t=\E[Z^i_t]$ and $\bm_t=(m^i_t)_{i\in\zz^d}$. A simple computation (using one more time
Lemma \ref{tlt}) gives us, for all $i\in\zz^d$,
$$
m^i_t=\mu_i t + \intot (2d+1)^{-1} \sum_{j\rightarrow i} \varphi(t-s) m^j_s ds.
$$
Using the vector formalism, this rewrites
$\bm_t=\bmu t + \intot \varphi(t-s) (A \bm_s) ds$. We furthermore know (from Theorem \ref{existence uniqueness Y})
that for all $t\geq 0$, $\sum_{i\in\zz^d} 2^{-|i|}m^i_t<\infty$.
Applying Lemma \ref{vectconv}, we see that $m^i$ is of class $C^1$ on $[0,\infty)$ for each $i\in\zz^d$, that
\begin{align}\label{jab1}
\bm_t'=\bmu + \intot \varphi(t-s) A \bm'_s ds
\end{align}
and that
\begin{align}\label{jab2}
\bm_t'= \Big(I + \sum_{n\geq 1} A^n \intot \varphi^{\star n}(s)ds \Big) \bmu.
\end{align}
Lemma \ref{vectconv} also tells us that $u_t=\sup_{i\in\zz^d} \sup_{[0,t]}(m^i_s)'$ 
is locally bounded, which of course implies that $\sup_{i\in\zz^d} \sup_{[0,t]} m^i_s$ is also locally bounded
(because $m^i_0=0$ for all $i\in\zz^d$), and that
\begin{align}\label{jab5}
u_t \leq C + \intot \varphi(t-s) u_s ds.
\end{align}
We introduce the martingales, for $i\in \zz^d$, (we use a tilde for compensation),
\begin{align*}
M^{i}_t=\intot\int_0^\infty \indiq_{\big\{z \leq \mu_i + (2d+1)^{-1}\sum_{j\rightarrow i} 
\int_0^{s-}\varphi(s-u)dZ_u^{j}\big\}}\tilde\pi^i(ds\,dz)
\end{align*}
and observe as usual that $[M^i,M^j]_t=0$ when $i\ne j$ 
(because these martingales a.s. never jump at the same time) while
$[M^i,M^i]_t=Z^i_t$. We finally introduce $U^i_t=Z^i_t-m^i_t$, the vectors  $\bU_t=(U^i_t)_{i\in\zz^d}$ and
$\bM_t=(M^i_t)_{i\in\zz^d}$ and observe that
\begin{align}\label{jab3}
\bU_t = \bM_t +  \intot \varphi(t-s) A\bU_s ds.
\end{align}
Indeed, for every $i\in\zz^d$, using Lemma \ref{tlt} and the equation satisfied by $m^i_t$, we find
\begin{align*}
U^i_t=&M^i_t+ \intot (2d+1)^{-1} \sum_{j\rightarrow i} \varphi(t-s) (Z^j_s-m^j_s) ds
= M^i_t + \intot \varphi(t-s) (A \bU_s)_i ds. 
\end{align*}
Equation \eqref{jab3} can be solved as usual as
\begin{align}\label{jab4}
\bU_t = \Big(\bM_t + \sum_{n\geq 1} \intot \varphi^{\star n}(t-s) A^n \bM_s ds \Big).
\end{align}
Finally, we easily check that $v_t=\sup_{i\in\zz^d}\sup_{[0,t]} \E[|U^i_s|]$ is locally bounded
(because $\E[|U^i_t|]\leq \E[|Z^i_t|]+m^i_t \leq 2m^i_t$) and
satisfies (start from \eqref{jab3}, use that $\E[|M^i_t|]\leq \E[[M^i,M^i]_t]^{1/2}=\E[Z^i_t]^{1/2}=(m^i_t)^{1/2}
\leq(\intot u_s ds)^{1/2}$) and that $A$ is stochastic)
\begin{align}\label{jab6}
v_t \leq \Big(\intot u_s ds\Big)^{1/2} + \intot \varphi(t-s)v_sds.
\end{align}

{\it Step 2.} Here we prove that there is a constant $C$ such that for all $i\in\zz^d$, $(m^i_t)'\leq C$
(and thus also $m^i_t \leq Ct$). This follows from 
\eqref{jab5}, which implies that $u_t\leq C + \Lambda u_t$, whence $u_t \leq C/(1-\Lambda)$.

\vip

{\it Step 3.} For all $i\in\zz^d$, $(m^i_t)' \sim (Q_\Lambda\bmu)_i$,
whence also $m^i_t \sim (Q_\Lambda\bmu)_it$, as $t\to \infty$.
Indeed, starting from \eqref{jab2}, using the monotone convergence theorem and that 
$\int_0^\infty \varphi^{\star n}(s)ds=(\int_0^\infty \varphi(s)ds)^n=\Lambda^n$,
$$
\lim_{t\to \infty} (m^i_t)'=\Big(\Big(I +\sum_{n\geq 1}  \Lambda^n A^n \Big)\bmu\Big)_i= (Q_\Lambda\bmu)_i.
$$

{\it Step 4.} There is a constant $C$ such that for all $i\in\zz^d$, all $t\geq 0$,
$\E[|U^i_t|]\leq C t^{1/2}$. Indeed, this follows from \eqref{jab6} and Step 2, which imply that
$v_t\leq Ct^{1/2}+\Lambda v_t$, whence $v_t \leq Ct^{1/2}/(1-\Lambda)$.

\vip

{\it Step 5.} The conclusion follows immediately, writing
$$
\E\Big[\Big| \frac{Z^i_t}{t} - (Q_\Lambda\bmu)_i\Big| \Big]\leq \E\Big[\Big| \frac{U^i_t}{t}\Big|\Big]
+ \Big| \frac{m^i_t}{t} - (Q_\Lambda\bmu)_i \Big|,
$$
which tends to $0$ as $t\to \infty$ by Steps 3 and 4.
\end{preuve}

We now turn to the supercritical case.

\begin{preuve} {\it of Theorem \ref{nn2}.} 
We consider $m$ (not to be confused with $\bm$)
the unique solution to $m_t=\mu t + \intot \varphi(t-s)m_s ds$, where
$\mu$ is the mean value defined by \eqref{condimu}. This equation is studied in details 
in Lemma \ref{Glemlim}: with $\alpha_0$ and $a_0$ defined in the statement, we have
$m_t \sim a_0 e^{\alpha_0 t}$ and $m'_t \sim a_0 \alpha_0 e^{\alpha_0 t}$ as $t\to \infty$,
as well as $\Gamma(t)=\sum_{n\geq 1} \varphi^{\star n}(t) \sim (a_0 \alpha_0^2/\mu) e^{\alpha_0 t}$
and $\Upsilon(t)=\intot \Gamma(s)ds \sim (a_0 \alpha_0/\mu) e^{\alpha_0 t}$.

\vip

{\it Step 1.}
We adopt the notation introduced in Step 1 of the proof of Theorem \ref{nn1}.

\vip
{\it Step 2.} 
Here we check that that there is $C$ such that for all $i\in\zz^d$,
$(m^i_t)' \leq C e^{\alpha_0t}$ (and thus $m^i_t \leq C e^{\alpha_0t}$).
This follows from \eqref{jab5}, which tells us that $u_t=\sup_{i\in\zz^d} (m^i_t)'$ is locally bounded
and satisfies $u_t \leq C + \intot \varphi(t-s)u_s ds$. Setting $h_t=u_t-\intot \varphi(t-s)u_s ds$,
we see that $h_t$ is locally bounded (from above and from below), because $u$ is locally bounded and $\varphi$
is locally integrable. We furthermore have $u=h+u\star \varphi$. Applying Lemma \ref{Glemlim}-(e),
we deduce that $u=h+h\star \Gamma$. But $h$ is bounded from above by $C$. Consequently,
$u \leq C + C\star \Gamma=C(1+\Upsilon)$, where $\Upsilon$ was defined in Lemma \ref{Glemlim}.
The conclusion follows from Lemma \ref{Glemlim}-(b).

\vip

{\it Step 3.} We now show that for all $i\in\zz^d$, $(m^i_t)'\sim (m_t)'$ (whence $m^i_t\sim m_t$) as $t\to \infty$.
Let us fix $i\in\zz^d$ and set $r^i_t=(m^i_t)'- (m_t)'$, which satisfies
$r_t^i = \mu_i -\mu  + \sum_{n\geq 1} ((A^n\bmu)_i-\mu) \intot \varphi^{\star n}(s)ds$. 
We know from Lemma \ref{mass}-(ii) that $\eta_n=(A^n\bmu)_i-\mu$ tends to $0$ as $n\to \infty$.
Consequently, Lemma \ref{Glemlim}-(d) tells us that $e^{-\alpha_0 t} 
\sum_{n\geq 1} ((A^n\bmu)_i-\mu) \intot \varphi^{\star n}(s)ds$ tends to $0$ as $t\to \infty$.
Hence $e^{-\alpha_0 t} r_t^i \to 0$ as $t\to \infty$. Using finally that $m'_t\sim a_0\alpha_0 e^{\alpha_0 t}$
as $t\to\infty$, we conclude that $(m^i_t)'/m_t'=1+r^i_t/m_t' \sim 1+ (a_0\alpha_0)^{-1} e^{-\alpha_0 t} r^i_t \to 1$
as desired.

\vip

{\it Step 4.} Here we check that for every $i\in\zz^d$, $e^{-\alpha_0t}\E[ |U^i_t|]$ tends to $0$ as $t\to\infty$.
We start from \eqref{jab4} to write
$$
|U^i_t| \leq |M^i_t| + \sum_{n\geq 1} \intot \varphi^{\star n}(t-s) |(A^n \bM_s)_i|ds.
$$
But $\E[|M^i_t|]\leq \E[Z^i_t]^{1/2}=(m^i_t)^{1/2} \leq C e^{\alpha_0t/2}$ by Step 2
and $\E[(A^n \bM_t)_i^2]=\sum_j (A^n(i,j))^2 m^j_s
\leq C e^{\alpha_0t} \sum_j (A^n(i,j))^2 \leq C e^{\alpha_0t} \e_n$ by Lemma \ref{mass}-(i), with $\e_n\to 0$ as 
$n\to \infty$. Consequently,
$$
e^{-\alpha_0t}\E[|U^i_t|]\leq C e^{-\alpha_0t/2} + C e^{-\alpha_0t}\sum_{n\geq 1} \e_n^{1/2} \intot \varphi^{\star n}(t-s) 
e^{\alpha_0s/2} ds.
$$
Lemma \ref{Glemlim}-(d) allows us to conclude.

\vip

{\it Step 5.} The conclusion follows, writing 
$$
\E\Big[\Big| \frac{Z^i_t}{a_0e^{\alpha_0t}} - 1\Big| \Big]\leq \E\Big[\Big| \frac{U^i_t}{a_0e^{\alpha_0t}}\Big|\Big] 
+ \Big| \frac{m^i_t}{a_0e^{\alpha_0t}} - 1\Big|,
$$
and using Steps 3, 4, and that $ m_t \sim a_0e^{\alpha_0 t}$ by Lemma \ref{Glemlim}-(a).
\end{preuve}

\subsection{Study of an impulsion}

Here we want to study how an {\it impulsion} at time $0$ at $i=0$ propagates.
To this end, we work under Assumption \ref{tic} with $\mu_i=0$ for all $i\in\zz^d$, 
but we assume that $Z^0$ has a jump at time $0$. 
Such a study is of course important: it allows us to measure, in some sense, the range of the interaction.

\vip

We first define precisely the process under study.

\begin{defi}\label{dfimp}
We work under Assumption \ref{tic}-(i)-(ii)
and consider a family $(\pi^i(ds \, dz),i\in\zz^d)$ 
of i.i.d. $(\cF_t)_{t\geq 0}$-Poisson measures on $[0,\infty) \times [0,\infty)$ with intensity measure $dsdz$.
We say that a family $(Z^i_t)_{i\in\zz^d,t\geq 0}$ of $(\cF_t)_{t\geq 0}$-adapted counting processes 
is an impulsion Hawkes process if
$$
\forall\; i \in \zz^d, \quad Z^i_t = \intot \int_0^\infty \indiq_{\displaystyle \big\{z 
\leq \sum_{j\rightarrow i} (2d+1)^{-1}[\int_0^{s-} \varphi(s-u)dZ^j_u +\varphi(s)\indiq_{\{j=0\}}]  \big\}}
\pi^i(ds\,dz).
$$
\end{defi}

As said previously, the term $\sum_{j\rightarrow i}\varphi(s)\indiq_{\{j=0\}}$ is interpreted as an {\it excitation}
due to a {\it forced}
jump of $Z^0$ at time $0$: simply rewrite it as $\indiq_{\{0\rightarrow i\}} \int_0^{s-} \varphi(s-u)\delta_u(ds)$.

\vip

The following proposition is easy.

\begin{prop}
Adopt the assumptions and notation of Definition \ref{dfimp}. There exists a pathwise
unique impulsion Hawkes process $((Z^i_t)_{i\in\zz^d,t\geq 0}$ 
such that $\sum_{i\in\zz^d} \E[Z^i_t]<\infty$ for all $t\geq 0$.
\end{prop}

\begin{proof} The proof resembles much that of Theorem \ref{existence uniqueness Y}, so we only sketch it.
We start with uniqueness and thus consider two impulsion Hawkes processes $(Z^i_t)_{i\in\zz^d,t\geq 0}$
and $(\tZ^i_t)_{i\in\zz^d,t\geq 0}$ such that $\sum_{i\in\zz^d} \E[Z^i_t+\tZ^i_t]<\infty$.
We set $\Delta^i_t=\intot |d(Z^i_s-\tZ^i_s)|$, $\delta^i_t=\E[\Delta^i_t]$ 
and $\delta_t=\sum_{i\in\zz^d} \delta^i_t$ (which is locally bounded by assumption). 
We may check that $\delta^i_t \leq (2d+1)^{-1}\intot \sum_{j\rightarrow i} \varphi(t-s)\delta^j_s ds$ exactly as in 
the proof 
of Theorem \ref{existence uniqueness Y}.
Summing in $i$ and recalling that each site has $2d+1$ neighbours, we find
$\delta_t \leq \intot \varphi(t-s)\delta_s ds$. Lemma \ref{grrrr}-(i) tells us that $\delta_t=0$
for all $t$, whence pathwise uniqueness.

\vip

Existence follows from a Picard iteration. Let us only check an {\it a priori} estimate 
implying that  $\sum_{i\in\zz^d} \E[Z^i_t]<\infty$ for all $t\geq 0$. Set $m^i_t=\E[Z^i_t]$ 
and $m_t=\sum_{i\in\zz^d} m^i_t$. A direct computation using Lemma \ref{tlt}
shows that $m_t^i = (2d+1)^{-1}\sum_{j\rightarrow i} [\intot \varphi(t-s) m^j_s ds + \indiq_{\{j=0\}}\intot\varphi(s)ds]$.
Summing in $i$, we find $m_t = \intot \varphi(t-s) m_s ds + \intot\varphi(s)ds$.
Using Lemma \ref{grrrr}-(i), that $\varphi$ is locally integrable (and that $\intot\varphi(s)ds$ is locally
bounded), we deduce that $\sup_{[0,T]} m_t \leq C(T,\varphi)$ as desired.
\end{proof}

We next compute the probability of the extinction event. Point 1 is a noticeable
property that makes the result very easy and precise.

\begin{thm}\label{ex} Adopt the assumptions and notation of Definition \ref{dfimp} and consider the 
impulsion Hawkes process $((Z^i_t)_{i,\in\zz^d,t\geq 0}$. 

\vip

1. The process $Z_t=\sum_{i\in\zz}Z^i_t$ is a scalar impulsion Hawkes process with excitation
function $\varphi$. In other words, $(Z_t)_{t\geq 0}$ is a counting process with compensator
$A_t=\intot \lambda_s ds$, where
$$
\lambda_t=\varphi(t) + \int_0^{t-} \varphi(t-s)dZ_s.
$$

\vip

2. We introduce the extinction event defined by $\Omega_e=\{\lim_{t\to \infty} \sum_{i\in\zz^d}Z^i_t<\infty\}$.
Setting $\Lambda = \int_0^\infty \varphi(s)ds$, we have
(i) $\Pr(\Omega_e)=0$ if $\Lambda=\infty$;
(ii) $\Pr(\Omega_e)=\exp(-\gamma_\Lambda \Lambda)$  if $\Lambda \in (1,\infty)$, where
$\gamma_\Lambda\in (0,1)$ is characterised by $\gamma_\Lambda \Lambda+\log(1-\gamma_\Lambda)=0$;
(iii) $\Pr(\Omega_e)=1$ if $\Lambda\in(0,1]$.
\end{thm}

Of course, we can sometimes use this theorem, by a simple comparison argument, if $\varphi_{ji}$ depends on $i,j$.
For example, to guarantee non-extinction with probability one, it suffices that all the
$\varphi_{ji}$ are bounded below by some $(2d+1)^{-1}\varphi$ such that $\int_0^\infty \varphi(s)ds=\infty$.

\begin{proof}
Point 1 is immediate: the compensator of the counting process $(Z_t)_{t\geq 0}$ is
\begin{align*}
A_t=& 
\intot \sum_{i\in\zz^d} \sum_{j\rightarrow i} (2d+1)^{-1}[\int_0^{s-} \varphi(s-u)dZ^j_u +\varphi(s)\indiq_{\{j=0\}}] ds\\
=&\intot\sum_{j\in\zz^d} [\int_0^{s-} \varphi(s-u)dZ^j_u +\varphi(s)\indiq_{\{j=0\}}] ds \\
=& \intot [\varphi(s) + \int_0^{s-} \varphi(s-u)dZ_u]ds.
\end{align*}

We next prove point 2. It is well-known Folklore that a scalar impulsion
Hawkes process can be related to a Poisson Galton-Watson process with Poisson$(\Lambda)$
reproduction law, but we give a direct proof for the sake of completeness.
If $\Lambda=\infty$, it suffices to note that
$\Pr(\Omega_e)=\Pr(Z_\infty < \infty)= \Pr(A_\infty<\infty) \leq \Pr(\int_0^\infty \varphi(s)ds<\infty)=0$.
When $\Lambda<\infty$, we introduce the martingale, for $\gamma\in (0,1)$,
$N^\gamma_t= - \gamma (Z_t- A_t)= - \gamma Z_t +
\gamma\int_0^t \varphi(s)ds +\gamma \int_0^t \varphi(t-s) Z_s ds$ by Lemma \ref{tlt}. 
We denote by $M^\gamma_t=\cE(N^\gamma)_t=e^{\gamma A_t} \prod_{s\leq t} (1-\gamma \Delta Z_s)$ 
its Dol\'eans-Dade exponential, see Jacod-Shiryaev \cite[Chapter 1, Section 4f]{js}. Since
$Z$ is a counting process, we see that
$$
M^\gamma_t=\exp\Bigl(\gamma\int_0^t \varphi(s)ds +\gamma\int_0^t \varphi(t-s) Z_s ds + \log(1-\gamma)Z_t \Bigr).
$$
If $\gamma \Lambda + \log(1-\gamma)\le 0$ then $M^\gamma$ is bounded
(because $\gamma\int_0^t \varphi(s)ds +\gamma\int_0^t \varphi(t-s) Z_s ds + \log(1-\gamma)Z_t
\leq \gamma \Lambda + (\gamma\Lambda+\log(1-\gamma))Z_t \leq \gamma \Lambda$), and thus converges
in $L^1$. Consequently, $\E[M^\gamma_\infty]=1$ and (one easily verifies, using that
$Z$ is non-decreasing, that $\lim_{t\to \infty} 
\int_0^t \varphi(t-s) Z_s ds=\Lambda Z_\infty$)
$$
\E\Bigl[\exp\Bigl(\gamma \Lambda + \bigl(\log(1-\gamma)+\gamma\Lambda\bigr)Z_\infty\Bigr)\Bigr]=1.
$$
But for all $x>0$, there is a unique $\gamma(x) \in (0,1)$ such that
$\gamma(x)\Lambda+\log(1-\gamma(x))=-x$, whence
$$
\E\Bigl[\exp\Bigl(-x Z_\infty\Bigr)\Bigr]= \exp(-\Lambda \gamma(x)).
$$
Consequently,
$$
\Pr(\Omega_e) = \Pr(Z_\infty<\infty) = \lim_{x \to 0+} \E\Bigl[\exp\Bigl(-x Z_\infty\Bigr)\Bigr]
= \exp(-\Lambda \gamma(0+)).
$$
If $\Lambda \in (0,1]$, we see that $\gamma(0+)=0$, so that $\Pr(\Omega_e)=1$. If now
$\Lambda>1$, $\gamma(0+)$ is the unique solution in $(0,1)$ to 
$\gamma(0+)\Lambda+\log(1-\gamma(0+))=0$.
\end{proof}

We next study more deeply, in the super-critical case, how the impulsion propagates.
Unfortunately, the computations are really tedious: we decided to restrict ourselves to a
particular case ($\varphi$ is an exponential function) 
where some computations are explicit. We believe that the result below
can be extended to a general class of functions $\varphi$, but a difficult technical lemma
is required.

\begin{thm}\label{propa}
Work under Assumption \ref{tic}-(i)-(ii), with $\varphi(t)=ae^{-bt}$, for some $a>b>0$.
Consider the impulsion Hawkes process $((Z^i_t)_{i\in\zz^d,t\geq 0}$. Since $\int_0^\infty \varphi(s)ds=a/b>1$,
we know from Theorem \ref{ex} that $\Pr(\Omega_e)\in(0,1)$.
We set $\alpha_0=a-b$ (for which $\cL_\varphi(\alpha_0)=1$) and we recall that the Gaussian density 
$p_t(x)$ is defined by \eqref{noyauGauss}.

\vip

(i) There are some constants $C>0$ and $t_0>0$ and a random variable $H\geq 0$
such that for all $i \in \zz^d$, all $t\geq t_0$,
$$
\E\big[\big| Z^i_t - H p_{at}(i) e^{\alpha_0 t}  \big| \big] \leq \frac{Ce^{\alpha_0 t}}{t^{d/2+1/3}}.
$$

(ii) For all $x \in\rr^d$,  $t^{d/2}e^{-\alpha_0 t} Z^{\lfloor x  t^{1/2} \rfloor}_t \to  H p_a(x)$
in probability as $t\to\infty$. Here, we define the ``integer part'' $\lfloor y \rfloor$ of 
$y=(y_1,\dots,y_d)\in\rr^d$, by
$\lfloor y \rfloor=(\lfloor y_1 \rfloor,\dots,\lfloor y_d \rfloor)\in\zz^d$.

\vip

(iii) The random variable $H$ is positive on the event $\Omega_e^c$.

\end{thm}

This result describes quite precisely how an impulsion propagates. 
Conditional on non extinction,
and the process $(Z^i_t)_{i\in\zz^d}$ resembles a Gaussian profile, with height $t^{-d/2} H e^{\alpha_0 t}$
and radius $\sqrt t$, for some positive random variable $H$.

\vip

Compared to the previous result (Theorem \ref{nn2}), the growth is only very slightly slower: 
a single impulsion at the site $0$ produces a growth in $t^{-d/2}e^{\alpha_0 t}$, while 
we have $e^{\alpha_0 t}$ when all the sites are regularly excited (as is e.g. the case when $\mu_i=1$ for all $i$).

\vip

Finally, it is important to note that, even if the growth ``near 0'' of the process is very fast (exponential),
the spatial propagation is quite slow (of order $\sqrt t$).

\vip

We start with some preliminary computations.

\begin{lem}\label{prelim}
Adopt the notation and assumptions of Theorem \ref{propa}. Introduce also, for $i\in\zz^d$ and $t\geq 0$, 
$\Gamma(i,t)=\sum_{n\geq 1} A^n(0,i) \varphi^{\star n} (t)$.

\vip

(i) For all $n\geq 1$, all $t\geq 0$, $\varphi^{\star n}(t) = (at)^{n-1} e^{-bt} / (n-1)!$. 

\vip

(ii) For all $t\geq 0$, $\sum_{i\in\zz^d}\Gamma(i,t) = e^{\alpha_0 t}$.

\vip

(iii) There is $C$ such that for all $t\geq 0$, $\sum_{i\in\zz^d}|i|^2\Gamma(i,t) \leq C (1+t) e^{\alpha_0 t}$.

\vip

(iv) There are some constants $C$ and $t_0>0$ such that for all $t\geq t_0$, all $i\in\zz^d$,
\begin{align}
&\big|\Gamma(i,t)-p_{a t}(i)e^{\alpha_0 t} \big| \leq  \frac{Ce^{\alpha_0 t}}{t^{d/2+1/3}}, \label{ss1}\\
&\big|\intot\Gamma(i,s)ds - \frac{1}{\alpha_0}p_{a t}(i)e^{\alpha_0 t} \big|\leq  \frac{Ce^{\alpha_0 t}}{t^{d/2+1/3}}
\label{ss2}
\end{align}
\end{lem}

\begin{proof}
Point (i) is well-known and can be checked recursively. 
Using that $A^n$ is a stochastic matrix, we see that $\sum_{i\in\zz^d}\Gamma(i,t)=\sum_{n\geq 1} \varphi^{\star n} (t)$.
Hence (ii) follows from (i).

\vip

Next, we recall that $A$ is the transition matrix of a symmetric random walk on $\zz^d$ (with bounded jumps),
so that there is a constant $C$ such that for all $n\geq 0$,
$\sum_{i\in\zz^d}|i|^2 A^n(0,i) \leq C n$ (its variance at time $n$ is of order $n$). Consequently,
$$
\sum_{i\in\zz^d} |i|^2 \Gamma(i,t) \leq C \sum_{n\geq 1} n \varphi^{\star n} (t)= 
C \sum_{n\geq 1} \frac{n (at)^{n-1} e^{-bt}}{(n-1)!} = C\big[ \sum_{n\geq 1} \frac{ (at)^{n-1} e^{-bt}}{(n-1)!}
+   \sum_{n\geq 2} \frac{ (at)^{n-1} e^{-bt}}{(n-2)!}  \big].
$$
This is easily computed: it gives $Ce^{\alpha_0 t}[1+at]$.

\vip

Point (iv) is more complicated. First, we need the Gaussian approximation of $A^n(0,i)$ given by \eqref{ltcl}.
We will also need the following result, which can be found e.g. in \cite[Lemma 9-(d)]{fg}
(plus the fact that for $x\in (0,1)$, $(x+1)\log(x+1)-x \geq x^2/4$): 
for any $\lambda>0$, for $X$ a Poisson$(\lambda)$-distributed random variable, for any $x\in (0,1)$,
\begin{align}\label{conpoi}
\Pr(|X-\lambda|\geq \lambda x ) \leq 2\exp(-\lambda x^2/4).
\end{align}
We now turn to our problem. Observing that $\sum_{n\geq 1} \varphi^{\star n} (t) = e^{\alpha_0 t}$, we write
\begin{align*}
\Delta(i,t)=\big|\Gamma(i,t)-p_{at}(i)e^{\alpha_0 t} \big| = \Big|\sum_{n\geq 1} \varphi^{\star n} (t) 
(A^n(0,i) -p_{at}(i)) \Big|.
\end{align*}
We now assume that $t$ is large enough so that $p_{t/a}(i)\leq 1$ for all $i$ and we write
\begin{align*}
\Delta(i,t)\leq& \sum_{n\geq 1} \varphi^{\star n} (t) \indiq_{\{|(n-1)-at|>(at)^{2/3}  \}} + 
\sum_{n\geq 1} \varphi^{\star n} (t)\indiq_{\{|(n-1)-at|\leq (at)^{2/3}  \}} \big|A^n(0,i) -p_{n}(i) \big|\\
&+ \sum_{n\geq 1} \varphi^{\star n} (t)\indiq_{\{|(n-1)-at|\leq (at)^{2/3}  \}} \big|p_n(i) -p_{at}(i) \big|\\
=&\Delta_1(i,t)+\Delta_2(i,t)+\Delta_3(i,t).
\end{align*}
First, using point (i) and \eqref{conpoi} (with $\lambda=at$ and $x=(at)^{-1/3}$), 
we see that (if $t$ is large enough
so that $at>1$)
\begin{align*}
\Delta_1(i,t)\leq& e^{\alpha_0 t}\sum_{n\geq 1} e^{-at}\frac{(at)^{n-1}}{(n-1)!} \indiq_{\{|(n-1)-at|>(at)^{2/3}  \}}
\leq 2 e^{\alpha_0 t} e^{-(at)^{1/3}/4 }.
\end{align*}
We next use \eqref{ltcl} and assume that $t$ is large enough so that $|(n-1)-at|\leq (at)^{2/3}$
implies $n \geq at/2$:
\begin{align*}
\Delta_2(i,t)\leq& C\sum_{n\geq 1} n^{-(d+2)/2} \varphi^{\star n} (t)\indiq_{\{|(n-1)-at|\leq (at)^{2/3}  \}}
\leq C(at)^{-(d+2)/2}\sum_{n\geq 1} \varphi^{\star n} (t) \leq C t^{-(d+2)/2} e^{\alpha_0 t}.
\end{align*}
Finally, we observe that $|\partial_t p_t(x)| \leq C t^{-d/2-1}$, so that,
if $t$ is sufficiently large,  $|(n-1)-at|\leq (at)^{2/3}$ implies $|p_n(i) -p_{at}(i)| \leq C t^{-d/2-1/3}$.
Consequently, 
$$
\Delta_3(i,t) \leq C  t^{-d/2-1/3}\sum_{n\geq 1} \varphi^{\star n} (t) \leq C  t^{-d/2-1/3}e^{\alpha_0 t}.
$$
We have proved that there are $C$ and $t_0$ such that for all $t\geq t_0$, all $i\in\zz^d$,
$$
\big|\Gamma(i,t)-p_{at}(i)e^{\alpha_0 t} \big|\leq  2 e^{\alpha_0 t} e^{-(at)^{1/3}/4 }
+C t^{-(d+2)/2} e^{\alpha_0 t}+C  t^{-d/2-1/3}e^{\alpha_0 t}
\leq C  t^{-d/2-1/3}e^{\alpha_0 t},
$$
which is \eqref{ss1}. 

\vip

It remains to deduce \eqref{ss2} from \eqref{ss1}.
We write
\begin{align*}
\delta(t,i)=& \Big|\intot \Gamma(i,s)ds -   p_{at}(i)e^{\alpha_0 t} /\alpha_0 \Big|\\
\leq & \int_0^{t-t^{1/2}}  \Gamma(i,s)ds + \int_{t-t^{1/2}}^t \big| \Gamma(i,s)- p_{as}(i)e^{\alpha_0 s}  \big| ds\\
&+  \int_{t-t^{1/2}}^t \big|p_{as}(i)e^{\alpha_0 s} - p_{at}(i)e^{\alpha_0 s} \big| ds
+ p_{at}(i)\big| \int_{t-t^{1/2}}^t e^{\alpha_0 s} ds -  e^{\alpha_0 t}/\alpha_0\big|\\
=& \delta_1(t,i)+ \delta_2(t,i)+ \delta_3(t,i)+ \delta_4(t,i).
\end{align*}
First, point (ii) implies that 
$$
\delta_1(t,i) \leq \int_0^{t-t^{1/2}} e^{\alpha_0 s}ds \leq C e^{\alpha_0(t-t^{1/2})}.
$$
Next, \eqref{ss1} tells us, if $t$ is sufficiently large (so that $t-t^{1/2}\geq t_0$ and 
$t-t^{1/2}\geq t/2$), that
$$
\delta_2(t,i) \leq C \int_{t-t^{1/2}}^t s^{-d/2-1/3} e^{\alpha_0 s}ds \leq C t^{-d/2-1/3}e^{\alpha_0 t}.
$$
Recalling that $|\partial_t p_t(x)| \leq C t^{-d/2-1}$, we get (still for $t$ large enough so that
$t-t^{1/2}\geq t/2$)
$$
\delta_3(t,i) \leq C \int_{t-t^{1/2}}^t s^{-d/2-1}(t-s) e^{\alpha_0 s}ds \leq C t^{-d/2-1/2}e^{\alpha_0 t}.
$$
Finally, if $t$ is suffiently large, we can bound $p_{at}(i)$ by $1$ (for all $i$), whence
$$
\delta_4(t,i) \leq \alpha_0^{-1} |e^{\alpha_0t}-e^{\alpha_0(t-t^{1/2})}  - e^{\alpha_0t}|\leq C e^{\alpha_0(t-t^{1/2})}.
$$
All in all, we have proved that for all $t$ large enough, all $i\in\zz^d$,
\begin{align*}
\Big|\intot \Gamma(i,s)ds -   p_{at}(i)e^{\alpha_0 t} /\alpha_0 \Big|
\leq C e^{\alpha_0(t-t^{1/2})}+C t^{-d/2-1/3}e^{\alpha_0 t}+C t^{-d/2-1/2}e^{\alpha_0 t} \leq C t^{-d/2-1/3}e^{\alpha_0 t}
\end{align*}
as desired.
\end{proof}

We finally can give the

\begin{preuve} {\it of Theorem \ref{propa}.} We divide the proof in several steps.
\vip

{\it Step 1.} As usual, we write 
\begin{align*}
Z^{i}_t=\intot\int_0^\infty \indiq_{\big\{z \leq (2d+1)^{-1}\sum_{j\rightarrow i} 
\big[\int_0^{s-}\varphi(s-u)dZ_u^{j}+ \varphi(s) \indiq_{\{j =0\}}\big]\big\}}
\pi^i(ds\,dz),
\end{align*}
we set $m^i_t=\E[Z^i_t]$ and $\bm_t=(m^i_t)_{i\in\zz^d}$. A simple computation, using Lemma \ref{tlt}, shows that
$$
m^i_t=\intot (2d+1)^{-1} \sum_{j\rightarrow i} \varphi(t-s) m^j_s ds + \indiq_{\{0\rightarrow i\}}
(2d+1)^{-1}\intot \varphi(s)ds.
$$
Using the vector formalism and introducing $\bdelta=(\delta_i)_{i\in\zz^d}$ defined
by $\delta_i=\indiq_{\{i=0\}}$, this rewrites
$\bm_t= (A\bdelta)\intot \varphi(s)ds + \intot \varphi(t-s) (A \bm_s) ds$. 
Differentiating this formula (see Lemma \ref{vectconv} for the justification 
of a very similar differentiation), we find
$\bm_t'=(A\bdelta) \varphi(t) + \intot \varphi(t-s) A \bm'_s ds$,
which can be solved as (see Lemma \ref{vectconv} again)
$\bm_t'= \sum_{n\geq 1} \varphi^{\star n}(t) A^n  \bdelta$.
Hence for all $i\in\zz^d$, 
\begin{align}\label{rrr1}
(m^i_t)'=\sum_{n\geq 1} A^n(0,i) \varphi^{\star n}(t).
\end{align}
We introduce the martingales, for $i\in \zz^d$, (we use a tilde for compensation),
\begin{align*}
M^{i}_t=\intot\int_0^\infty \indiq_{\big\{z \leq (2d+1)^{-1}\sum_{j\rightarrow i} 
\big[\int_0^{s-}\varphi(s-u)dZ_u^{j}+ \varphi(s) \indiq_{\{j =0\}}\big]\big\}}
\tilde\pi^i(ds\,dz)
\end{align*}
and observe as usual that $[M^i,M^j]_t=0$ when $i\ne j$ 
(because these martingales a.s. never jump at the same time) while
$[M^i,M^i]_t=Z^i_t$. We will use several times that for any family 
$(\alpha_i)_{i\in\zz^d}$,
\begin{align}\label{wwust}
\E\big[\big(\sum_{i\in\zz^d} \alpha_i M^i_t\big)^2\big] = \sum_{i\in\zz^d} \alpha_i^2 m^i_t.
\end{align}
We finally introduce $U^i_t=Z^i_t-m^i_t$, the vectors  $\bU_t=(U^i_t)_{i\in\zz^d}$ and
$\bM_t=(M^i_t)_{i\in\zz^d}$ and observe, exactly as in the proof of Theorem \ref{nn1}-Step 1, that
$\bU_t = \bM_t + A \intot \varphi(t-s) \bU_s ds$, whence 
$\bU_t = \Big(\bM_t + \sum_{n\geq 1} A^n \intot \varphi^{\star n}(t-s) \bM_s ds \Big)$ and thus,
for all $i\in\zz^d$,
\begin{align}\label{rrr2}
Z^i_t = m^i_t + M^i_t + \sum_{j\in\zz^d} \sum_{n\geq 1} \intot \varphi^{\star n} (t-s) A^n(i,j) M^j_s ds=m^i_t+M^i_t+W^i_t,
\end{align}
the last equality defining $W^i_t$. 

\vip

{\it Step 2.} Here we treat the terms $m^i_t$, $M^i_t$, and collect a few more information.
Recall that $\Gamma(t,i)=\sum_{n\geq 1} A^n(0,i)\varphi^{\star n}(t)$.
Starting from \eqref{rrr1}, we see that $m^i_t = \int_0^t  \Gamma(i,s)ds$ and 
deduce from Lemma \ref{prelim}-(iv) that there are $C$ and $t_0\geq 0$ such that
for all $i\in\zz^d$, all $t\geq t_0$,
$$
\big|m^i_t -\frac{1}{\alpha_0} p_{at}(i)e^{\alpha_0 t} \big|
\leq  \frac{Ce^{\alpha_0 t}}{t^{d/2+1/3}}.
$$
Next, Lemma \ref{prelim}-(ii)-(iii) imply that $\sum_{i \in \zz^d} m^i_t = \int_0^t e^{\alpha_0 s} ds
\leq C e^{\alpha_0 t}$ and $\sum_{i \in \zz^d} |i|^2 m^i_t \leq C \int_0^t (1+s)e^{\alpha_0 s} ds
\leq C(1+t) e^{\alpha_0 t}$. Finally, we observe that
$$
\E[|M^i_t|] \leq (m^i_t)^{1/2} \leq \Big(\sum_{j \in \zz^d} m^j_t\Big)^{1/2} \leq C e^{\alpha_0 t /2}.
$$

{\it Step 3.} We introduce 
$$
X= \sum_{j\in\zz^d}\int_0^\infty e^{-\alpha_0s}M^j_s ds 
$$
and show that there are $C>0$ and $t_0>0$ such that for all $i\in\zz^d$, all $t\geq t_0$,
$$
\E[|W^i_t- p_{at}(i) e^{\alpha_0 t} X |] \leq \frac{Ce^{\alpha_0 t}}{t^{d/2+1/3}}.
$$
We observe that since $A^n(i,j)=A^n(0,i-j)$, it holds that
$W^i_t= \sum_{j\in\zz^d} \int_0^{t} \Gamma(i-j,t-s) M^j_s ds$. We also 
introduce the auxiliary processes
\begin{align*}
\bW^i_t=& \sum_{j\in\zz^d} \int_0^{t^{1/2}} \Gamma(i-j,t-s) M^j_s ds,\\
\tW^i_t=& \sum_{j\in\zz^d} \int_0^{t^{1/2}} p_{a(t-s)}(i-j)e^{\alpha_0(t-s)} M^j_s ds,\\
\hW^i_t=& \sum_{j\in\zz^d} \int_0^{t^{1/2}} p_{at}(i)e^{\alpha_0(t-s)} M^j_s ds.
\end{align*}

{\it Step 3.1.} Here we show that $\E[|W^i_t - \bW^i_t|] \leq C \exp(\alpha_0 t - (\alpha_0/2)t^{1/2})$.
By definition of $\Gamma$ and using \eqref{wwust},
\begin{align*}
\E[|W^i_t - \bW^i_t|] \leq& \int_{t^{1/2}}^t \sum_{n\geq 1} \varphi^{\star n}(t-s)
\E\big[\big|\sum_{j\in\zz^d} A^n(i,j)M^j_s \big|\big] ds\\
\leq& \int_{t^{1/2}}^t \sum_{n\geq 1} \varphi^{\star n}(t-s)
\big(\sum_{j\in\zz^d} (A^n(i,j))^2 m^j_s \big)^{1/2} ds.
\end{align*}
Using that $A^n(i,j)$ is bounded by $1$ and that
$\sum_{i \in \zz^d} m^i_t \leq C e^{\alpha_0 t}$ (see Step 2), we see that 
$\sum_{j\in\zz^d} (A^n(i,j))^2 m^j_s \leq Ce^{\alpha_0 s/2}$. Next, 
the explicit expression of $\varphi^{\star n}$ (see Lemma \ref{prelim}-(i)) gives 
$\sum_{n\geq 1} \varphi^{\star n}(t-s)=e^{\alpha_0 (t-s)}$. We finally find
\begin{align*}
\E[|W^i_t - \bW^i_t|] \leq C \int_{t^{1/2}}^t e^{\alpha_0(t-s)}  e^{\alpha_0 s/2} ds
\leq C e^{\alpha_0 t}e^{-\alpha_0 t^{1/2}/2}.
\end{align*}

{\it Step 3.2.} We next check that $\E[|\bW^i_t-\tW^i_t|]\leq C t^{-d/2-1/3}e^{\alpha_0 t}$.
Using \eqref{wwust}, we get
\begin{align*}
\E[|\bW^i_t - \tW^i_t|] \leq& \int_0^{t^{1/2}} \Big[\sum_{j\in\zz^d} m^j_s \big( 
\Gamma(i-j,t-s)-p_{a(t-s)}(i-j)e^{\alpha_0(t-s)}\big)^2 \Big]^{1/2} ds.
\end{align*}
Using Lemma \ref{prelim}-(iv),
$|\Gamma(i-j,t-s)-p_{a(t-s)}(i-j)e^{\alpha_0(t-s)}| \leq  C (t-s)^{-d/2-1/3} e^{\alpha_0(t-s)}$
if $t-s$ is large enough (which is the case for all $s\in [0,t^{1/2}]$ if $t$ is large enough).
Since furthermore $\sum_{i \in \zz^d} m^i_t \leq C e^{\alpha_0 t}$ (see Step 2), we find
\begin{align*}
\E[|\bW^i_t - \tW^i_t|] \leq& C e^{\alpha_0 t}\int_0^{t^{1/2}} (t-s)^{-d/2-1/3}e^{- \alpha_0 s/2}ds.
\end{align*}
For $t$ large enough, we clearly have $(t-s)^{-d/2-1/3} \leq 2 t^{-d/2-1/3}$ for all $s\in [0,t^{1/2}]$, whence
\begin{align*}
\E[|\bW^i_t - \tW^i_t|] \leq& C t^{-d/2-1/3}  e^{\alpha_0 t} \int_0^{t^{1/2}} e^{-\alpha_0 s / 2}ds \leq
C t^{-d/2-1/3}  e^{\alpha_0 t} .
\end{align*}

{\it Step 3.3.} We now prove that $\E[|\tW^i_t-\hW^i_t|]\leq C t^{-d/2-1/2}e^{\alpha_0 t}$. As usual,
we start with
\begin{align*}
\E[|\tW^i_t - \hW^i_t|] \leq&  \int_0^{t^{1/2}} \Big(\sum_{j\in\zz^d} m^j_s \big(
p_{a(t-s)}(i-j)-p_{a t}(i) \big)^2 \Big)^{1/2} e^{\alpha_0(t-s)} ds.
\end{align*}
But an easy computation (using that $|\partial_t p_t(x)| \leq C t^{-d/2-1}$ and
$|\nabla_x p_t(x)| \leq C t^{-d/2-1/2}$) shows that for all $t>0$, all $h\in(0,t/2)$, all $x,y \in \rr^d$,
$|p_{t-h}(x-y)-p_t(x)| \leq C h t^{-d/2-1} + |y| t^{-d/2-1/2}$. Hence if $t$ is large enough so that 
$t-t^{1/2} \geq t/2$, we can write
\begin{align*}
\E[|\tW^i_t - \hW^i_t|] \leq& C \int_0^{t^{1/2}} \Big(\sum_{j\in\zz^d} m^j_s \big(s t^{-d/2-1} + |j| t^{-d/2-1/2}\big)^2 
\Big)^{1/2} e^{\alpha_0(t-s)} ds\\
\leq & C t^{-d/2-1/2} \int_0^{t^{1/2}} \Big(\sum_{j\in\zz^d} m^j_s\Big)^{1/2} e^{\alpha_0(t-s)} ds\\
&+  C t^{-d/2-1/2} \int_0^{t^{1/2}} \Big(\sum_{j\in\zz^d} |j|^2m^j_s\Big)^{1/2} e^{\alpha_0(t-s)} ds.
\end{align*}
Finally, we know from Step 2 that $\sum_{j\in\zz^d} m^j_s+\sum_{j\in\zz^d} |j|^2m^j_s \leq C(1+s)e^{\alpha_0 s}$, whence
\begin{align*}
\E[|\tW^i_t - \hW^i_t|] \leq& C t^{-d/2-1/2} \int_0^{t^{1/2}} (1+s)^{1/2}e^{\alpha_0 s/2} e^{\alpha_0(t-s)} ds
\leq C  t^{-d/2-1/2}e^{\alpha_0 t}.
\end{align*}

{\it Step 3.4} We finally verify that $\E[|\hW^i_t- p_{at}(i) e^{\alpha_0 t} X |] \leq 
C e^{\alpha_0 t - (\alpha_0/2) t^{1/2}}$.
We note that
$$
\hW^i_t- p_{at}(i) e^{\alpha_0 t} X=  p_{at}(i)e^{\alpha_0 t}\sum_{j\in\zz^d} \int_{t^{1/2}}^t e^{-\alpha_0 s} M^j_s ds.
$$
For $t$ large enough (not depending on $i$), 
we can bound $p_{at}(i)$ by $1$. Hence, we infer from \eqref{wwust} and 
the fact that $\sum_{j\in\zz^d} m^j_s \leq Ce^{\alpha_0 s}$ that
\begin{align*}
\E[|\hW^i_t- p_{at}(i) e^{\alpha_0 t} X |] \leq & C e^{\alpha_0 t} \int_{t^{1/2}}^t e^{-\alpha_0 s} 
\big(\sum_{j\in\zz^d} m^j_s \big)^{1/2} ds \leq C e^{\alpha_0 t} \int_{t^{1/2}}^\infty e^{-\alpha_0 s /2}  ds,
\end{align*}
from which the conclusion follows.

\vip

{\it Step 3.5.} Gathering Steps 3.1 to 3.4, we conclude that indeed, there are $C>0$ and $t_0>0$
such that for all $t\geq t_0$, all $i\in\zz^d$, $\E[|W^i_t- p_{at}(i) e^{\alpha_0 t} X |]
\leq C e^{\alpha_0 t} t^{-d/2-1/3}$.

\vip

{\it Step 4.} Define the random variable $H=\alpha_0^{-1}+X$. Recall \eqref{rrr2}
and write $Z^i_t - p_{at}(i) e^{\alpha_0 t} H = [m^i_t - \alpha_0^{-1}p_{at}(i) e^{\alpha_0 t}]
+ M^i_t + [W^i_t - p_{at}(i) e^{\alpha_0 t} X]$.
Gathering Steps 2 and 3,
we see that there are $C>0$ and $t_0>0$ such that for all $t\geq t_0$, all $i\in\zz^d$, 
$$
\E[|Z^i_t- p_{at}(i) e^{\alpha_0 t} H |] \leq \frac{Ce^{\alpha_0 t}}{t^{d/2+1/3}} + Ce^{\alpha_0 t /2} \leq 
\frac{Ce^{\alpha_0 t}}{t^{d/2+1/3}}.
$$
To prove that $H$ is nonnegative, it suffices to use the above inequality with $i=0$, divided by
$p_{at}(0) e^{\alpha_0 t}$. Recalling that $p_{at}(0)=ct^{-d/2}$ for some constant $c$, we deduce that
$\E[|H - Z^0_t e^{-\alpha_0 t}/p_{at}(0)|] \leq C t^{-1/3}$. Consequently, $H$ is the limit (in $L^1$)
of $Z^0_t e^{-\alpha_0 t}/p_{at}(0)$, and is thus nonnegative.
This ends the proof of (i).

\vip

{\it Step 5.} We now check (ii), which follows from (i): for $x\in \rr^d$, and $t\geq t_0$,
\begin{align*}
\E\big[\big|t^{d/2}e^{-\alpha_0 t} Z_t^{\lfloor x t^{1/2}\rfloor} - H p_a(x) \big|\big] \leq &
t^{d/2}e^{-\alpha_0 t}\E\big[\big|Z_t^{\lfloor x t^{1/2}\rfloor} - H e^{\alpha_0 t}p_{at}(\lfloor x t^{1/2} \rfloor) \big|\big] \\
&+ \E[H] \big|t^{d/2}p_{at}(\lfloor x t^{1/2} \rfloor) -p_a(x)\big|.
\end{align*}
The first term on the RHS is bounded, by (i), by $C t^{-1/3}$, which tends to $0$ as $t\to\infty$.
The second term also tends to $0$, simply because
$$
t^{d/2}p_{at}(\lfloor x t^{1/2} \rfloor)=p_a\big(t^{-1/2} \lfloor x t^{1/2}\rfloor \big) \to p_a(x) 
\quad \hbox{as $t\to\infty$.}
$$

{\it Step 6.} It remains to prove (iii). 
First note that $H = \lim_{t \rightarrow \infty}e^{-\alpha_0 t} Z_t$ in $L^1$, where $Z_t = \sum_{i \in \zz^d}Z_t^i$. 
Indeed, recalling \eqref{rrr2}
$$
Z_t = \sum_{i \in \zz^d}m_t^i+\sum_{i \in \zz^d}M_t^i+\sum_{i \in \zz^d}W_t^i.
$$
We have $\sum_{i \in \zz^d}m_t^i=\intot e^{\alpha_0 s} ds=\alpha_0^{-1}[e^{\alpha_0t}-1]$ by Step 2, 
\eqref{wwust} implies that $\E[(\sum_{i \in \zz^d}M_t^i)^2]=\sum_{i \in \zz^d}m_t^i$ and finally
$\sum_{i \in \zz^d}W_t^i=e^{\alpha_0 t}\int_0^t e^{-\alpha_0s}\sum_{i \in \zz^d}M_s^ids$, therefore $e^{-\alpha_0 t}Z_t$ 
converges to $\alpha_0^{-1}+0+X=H$ as $t\rightarrow \infty$ in $L^1$.
We also note that $\E[H]=\alpha_0^{-1}>0$ since $\E[X]=0$.

\vip

Next, we recall that by Theorem \ref{ex}-(i), $(Z_t)_{t\geq 0}$ is a scalar impulsion Hawkes process: its compensator
is given by $A_t=\intot [\varphi(s) + \int_0^{s-}\varphi(s-u)dZ_u]ds$. We claim that $(Z_t)_{t\geq 0}$ 
has the same law as $(\tZ_t)_{t\geq 0}$ built as follows: 

\vip

$\bullet$ consider a Poisson process $(N_t)_{t\geq 0}$ with intensity
$\varphi(t)dt$, observe that $N_\infty$ is Poisson$(\Lambda)$-distributed,
denote by $0<T_1<\dots<T_{N_\infty}$ its times of jump (we adopt the convention that $T_i=\infty$ for
$i > N_\infty$), 

$\bullet$ consider an i.i.d. family $(\tZ_t^k)_{t\geq 0}$ of scalar impulsion Hawkes process with same law as 
$(Z_t)_{t\geq 0}$,

$\bullet$ put $\tZ_t= N_t+ \sum_{i=1}^{N_\infty} \tZ_{t-T_k}^k \indiq_{\{t\geq T_k\}}$.

\vip

\noindent Indeed, $(\tZ_t)_{t\geq 0}$ is a counting process with compensator
$$
\intot [\varphi(s)+ \sum_{i\geq 1} 1_{\{s>T_k\}} (\varphi(s-T_k) + \int_0^{(s-T_k)-} \varphi(s-T_k-u) d\tZ^k_u )]ds
=\intot [\varphi(s) + \int_0^{s-} \varphi(s-u) d\tZ_u]ds.
$$

We define $\tH=\lim_{t\to \infty} e^{-\alpha_0t}\tZ_t$ and, for each $k\geq 1$, 
$\tH_k=\lim_{t\to \infty} e^{-\alpha_0t}\tZ_t^k$. We obviously have $\tH = \sum_{k=1}^{N_\infty} e^{-\alpha_0 T_k}\tH_k$.
Denoting by $p=\Pr(H=0)$ (which also equals $\Pr(\tH=0)$ and $\Pr(\tH_k=0)$ for all $k\geq 1$), we deduce
by independence of the family $(\tH_k)_{k\geq 1}$, that 
$p= \sum_{n \geq 1} \Pr(N_\infty=n) p^n$. Hence $p= \sum_{n \geq 1} e^{-\Lambda} \Lambda^np^n/n!= e^{-\Lambda(1-p)}$.
Since $\E[H]>0$, we cannot have $p =1$. Hence $p$ is the unique solution in $(0,1)$ 
to $p=e^{-\Lambda(1-p)}$. Recalling that $\Pr(\Omega_e)=\exp(-\gamma_\Lambda \Lambda)$ where $\gamma_\Lambda \in (0,1)$
is characterised by $\gamma_\Lambda \Lambda + \log(1-\gamma_\Lambda)=0$ by Theorem \ref{ex}-2-(ii), we 
easily check that $p=\Pr(\Omega_e)$.

\vip

By definition of $H$, we have $\Omega_e \subset \{H=0\}$. Since
$\Pr(\Omega_e)=\Pr(H=0)$ we conclude that a.s., $H>0$ on $\Omega_e^c$.
\end{preuve}

\section{Appendix: convolution equations}

We collect here some technical results about convolution equations.
We start with an identity of constant use in the paper.

\begin{lem}\label{tlt}
Let $\phi:[0,\infty)\mapsto \rr$ be locally integrable and let $\alpha:[0,\infty)\mapsto \rr$
have finite variations on compact intervals and satisfy $\alpha(0)=0$. Then for all $t\geq 0$,
$$
\intot \int_0^{s-} \phi(s-u) d\alpha(u) ds = 
\intot \int_0^{s} \phi(s-u) d\alpha(u) ds = 
\int_0^{t} \phi(t-s)\alpha(s)ds.
$$
\end{lem}

\begin{proof}
First, we clearly have that $\int_0^{s-} \phi(s-u) d\alpha(u)=\int_0^{s} \phi(s-u) d\alpha(u)$ 
for almost every $s\geq 0$, whence 
$\intot \int_0^{s-} \phi(s-u) d\alpha(u) ds=\intot \int_0^{s} \phi(s-u) d\alpha(u) ds$.
Using twice the Fubini theorem,
\begin{align*}
\intot \Big(\int_0^{s} \phi(s-u) d\alpha(u)\Big) ds=&\int_0^{t} \Big(\int_{u}^t \phi(s-u)ds\Big)d\alpha(u) \\
=& \int_0^{t}  \Big(\int_0^{t-u} \phi(v)dv \Big) d\alpha(u)\\
=& \int_0^{t}  \Big(\int_0^{t-v} d\alpha(u)   \Big) \phi(v)dv\\
=& \int_0^{t} \alpha(t-v)\phi(v)dv,
\end{align*}
from which the conclusion follows, using the substitution $s=t-v$.
\end{proof}

We carry on with a generalized Gr\"onwall-Picard lemma, which is more or less standard.

\begin{lem}\label{grrrr} Let $\phi:[0,\infty)\mapsto [0,\infty)$ be locally integrable
and $g:[0,\infty)\mapsto [0,\infty)$ be locally bounded.

\vip

(i) Consider a locally bounded nonnegative function $u$ such that for all $t\geq 0$,
$u_t \leq g_t+\intot \phi(t-s)u_sds$ for all $t\geq 0$. 
Then $\sup_{[0,T]}u_t \leq C_T \sup_{[0,T]} g_t$, for some constant
$C_T$ depending only on $T>0$ and $\phi$.

\vip

(ii) Consider a sequence of locally bounded nonnegative functions $u^n$ such that for all 
$t\geq 0$, all $n\geq 0$, $u_t^{n+1} \leq \intot \phi(t-s)u_s^nds$. Then
$\sup_{[0,T]}\sum_{n\geq 0}u_t^n \leq C_T$,
for some constant $C_T$ depending only on $T>0$, $u^0$ and $\phi$.

\vip

(iii) Consider a sequence of locally bounded nonnegative functions $u^n$ such that for all 
$t\geq 0$,
all $n\geq 0$, $u_t^{n+1} \leq g_t+\intot \phi(t-s)u_s^nds$. Then for all $T\geq 0$, 
$\sup_{[0,T]} \sup_{n\geq 0} u_t^n \leq C_T$,
for some constant $C_T$ depending only on $T>0$, $u^0$, $g$ and $\phi$.
\end{lem}

\begin{proof}
We start with point (i). Fix $T>0$ and consider $A>0$ such that
$\int_0^T\phi(s)\indiq_{\{\phi(s)\geq A\}}ds\leq 1/2$. Then for all $t\in[0,T]$,
\begin{align*}
u_t \leq g_t+\intot \indiq_{\{\phi(t-s)\leq A\}} \phi(t-s)u_s ds + 
\intot \indiq_{\{\phi(t-s)>A\}} \phi(t-s)u_s ds \leq 
g_t + A\intot u_s ds + \sup_{[0,t]} u_s /2.
\end{align*} 
from which we deduce that
$\sup_{[0,t]} u_s \leq 2 \sup_{[0,t]}g_s + 2A\intot u_s ds$. We then can apply
the standard Gr\"onwall Lemma to get $\sup_{[0,T]} u_s \leq 2 (\sup_{[0,T]}g_s ) e^{2AT}$.

\vip

To check point (iii), put $v^n_t=\sup_{k=0,\dots,n}u_t^k$. One easily checks 
that for all $n\geq 0$, $v^n_t \leq u^0_t+g_t + \intot \phi(t-s)v_s^nds$.
By point (i), $\sup_{[0,T]}v_t^n \leq C_T \sup_{[0,T]}(g_t+u^0_t)$. Letting 
$n$ increase to infinity concludes the proof.

\vip
 
Point (ii) follows from point (iii), since $v^n_t=\sum_{k=0}^n u^k_t$
satisfies $v^{n+1}_t \leq u^0_t + \intot \phi(t-s)v^n_s ds$.
\end{proof}

We next prove an easy well-posedness result for a general convolution equation.

\begin{lem}\label{expuni}
Let $h:\rr\mapsto [0,\infty)$ be Lipschitz-continuous and $\varphi : [0,\infty) \mapsto \rr$ 
be locally integrable. The equation
\begin{equation} \label{eqtostudy}
m_t = \int_0^t h\big(\int_0^s \varphi(s-u)dm_u\big)ds
\end{equation}
has a unique non-decreasing locally bounded solution. Furthermore, $m$ is of class $C^1$
on $[0,\infty)$.

\vip

If $h(x)=\mu+x$ for some $\mu>0$ and if $\varphi$ is nonnegative, Equation \eqref{eqtostudy}
rewrites as 
\begin{equation} \label{eqtostudy2}
m_t =  \mu t + \intot \varphi(t-s)m_s ds.
\end{equation}
\end{lem}

\begin{proof}
Let $m$ and $\widetilde m$ be two such solutions.
Since $h$ is Lipschitz-continuous,
\begin{align*}
v_t = \int_0^t\big|d(m_u-\widetilde m_u)\big| & \leq C \int_0^t ds \int_0^s 
|\varphi(s-u)|\big|d(m_u-\widetilde m_u)\big|=C\int_0^t |\varphi(t-u)| v_u du.
\end{align*}
The last inequality follows from Lemma \ref{tlt}. Lemma \ref{grrrr}-(i) allows
us to conclude that $v_t=0$ for all $t\geq 0$ (because $v_t\leq m_t+\tilde m_t$ is locally
bounded), whence $m_t=\tilde m_t$ for all $t\geq 0$. 
For the existence, we consider the sequence of non-decreasing functions
$m_t^0=0$ and $m_t^{n+1}=\int_0^th\big( \int_0^s\varphi(s-u)dm_u^{n}\big)ds$ for every $n\geq 0$.
We easily check that $m_t^{n+1}\leq h(0) t + |h|_{lip}\intot |\varphi(t-u)|m^n_udu$, so that
$\sup_{n\geq 0} m^n_t$ is locally bounded by Lemma \ref{grrrr}-(iii).
Setting, for $n\geq 0$,  $\delta^{n}_t=\intot |d(m^{n+1}_u-m^n_u)|$, one 
readily gets
$\delta^{n+1}_t\leq  |h|_{lip}\intot |\varphi(s-u)|\delta^n_u du$ for all $n\geq 0$. Lemma
\ref{grrrr}-(ii) thus implies that $\sum_{n\geq 0} \delta^{n}_t <\infty$. All this classically
implies the existence of a locally bounded non-decreasing $m$ such that for all $t\geq 0$, 
$\lim_n \intot |d(m_u-m^n_u)|=0$. Checking that $m$ solves \eqref{eqtostudy} is routine.

\vip

To prove that $m$ is $C^1$, we use the previous Picard Iteration.
One easily sees, by induction, that $m^n$ is $C^1$ for all $n\geq 0$ and that 
$(m^{n+1}_t)'=h(\intot \varphi(t-u)(m^n_u)'du)$ (indeed,
$t\mapsto \intot \varphi(t-u)(m^n_u)'du=\intot \varphi(u)(m^n_{t-u})'du$ is continuous
because $\varphi$ is locally integrable and because $(m^n_t)'$ is continuous by the inductive assumption).
Next, a direct computation shows that
$|(m^{n+1}_t)'-(m^n_t)'| \leq C  \intot|\varphi(t-u)||(m^n_u)'- (m^{n-1}_u)' |du$.
Using Lemma \ref{grrrr}-(ii), we deduce that the sequence $(m^n)'$ is Cauchy (for the uniform convergence
on compact time intervals). The conclusion classically follows.

\vip

The equivalence between \eqref{eqtostudy} and \eqref{eqtostudy2} in the linear case directly follows from
Lemma \ref{tlt}.
\end{proof}

We now investigate the large-time behaviour of $m_t$ in the linear case. We start with the
subcritical case.

\begin{lem}\label{Plemlim}
Consider $\mu>0$ and a function 
$\varphi:(0,\infty)\mapsto [0,\infty)$ such that $\Lambda=\int_0^\infty \varphi(s)ds<1$.
By Lemma \ref{expuni}, \eqref{eqtostudy2} has a unique non-decreasing
locally bounded solution $(m_t)_{t\geq 0}$, which is furthermore of class $C^1$. 
It holds that $m'_t \sim a_0$ and $m_t \sim a_0 t$ as $t\to \infty$, where $a_0=\mu /(1-\Lambda)>0$.
\end{lem}

\begin{proof} 
We rather use \eqref{eqtostudy}, which writes $m_t=\mu t + \intot \int_0^s \varphi(s-u)m'_udu ds$.
Differentiating this expression, we find
$m'_t=\mu+\intot \varphi(t-s)m'_s ds$. 
We first introduce $u_t=\sup_{[0,t]} m'_s$. We have $u_t \leq \mu + \Lambda u_t$, whence
$u_t \leq \mu/(1-\Lambda)$ for all $t\geq 0$ and thus  $\limsup_{t\to \infty} m'_t \leq \mu/(1-\Lambda)$.
We next introduce $v_t=\inf_{s\geq t} m'_s$, which is non-decreasing and thus has a limit $\ell\in (0,\infty]$. 
We have $v_t \geq \mu + v_{t/2} \int_0^{t/2} \varphi(s)ds \to \mu+ \Lambda \ell$ as $t\to \infty$.
Consequently $\ell \geq \mu+\Lambda\ell$, whence $\ell\geq  \mu/(1-\Lambda)$ 
and finally $\liminf_{t\to \infty} m'_t \geq \mu/(1-\Lambda)$. All this proves that
$m'_t \sim a_0$ and this implies that $m_t\sim a_0 t$.
\end{proof}

We now turn to the supercritical case.

\begin{lem}\label{Glemlim}
Consider $\mu>0$ and a function 
$\varphi:[0,\infty)\mapsto [0,\infty)$ such that $\Lambda=\int_0^\infty \varphi(s)ds \in(1,\infty]$.
By Lemma \ref{expuni}, \eqref{eqtostudy2} has a unique non-decreasing
locally bounded solution $(m_t)_{t\geq 0}$, which is of class $C^1$.
Assume furthermore that $t\mapsto \intot |d\varphi(s)|$ has at most polynomial growth.
Let $\Gamma(t)=\sum_{n\geq 1} \varphi^{\star n}(t)$ and $\Upsilon(t)=\intot \Gamma(s)ds$.

\vip

(a) There is a unique $\alpha_0>0$ such that $\cL_\varphi(\alpha_0)=1$. The function $\Gamma$
is locally bounded. Setting $a_0=\mu \alpha_0^{-2}(\int_0^\infty t \varphi(t)e^{-\alpha_0 t}dt)^{-1}$, we have,
as $t\to \infty$, 
$$
\Gamma(t) \sim (a_0\alpha_0^2/\mu) e^{\alpha_0 t}, \quad \Upsilon(t) \sim (a_0\alpha_0/\mu) e^{\alpha_0 t}, \quad 
m_t \sim a_0 e^{\alpha_0 t}, \quad m'_t \sim a_0\alpha_0 e^{\alpha_0 t}.
$$

\vip

(b) There are some constants $0<c<C$ such that for all $t\geq 0$, 
$ce^{\alpha_0 t} \leq \Gamma(t)+1 \leq Ce^{\alpha_0 t}$, 
$ce^{\alpha_0 t} \leq \Upsilon(t)+1 \leq Ce^{\alpha_0 t}$, 
$ce^{\alpha_0 t} \leq m_t+1\leq Ce^{\alpha_0 t}$ and $ ce^{\alpha_0 t} \leq m_t'+1\leq Ce^{\alpha_0 t}$.

\vip

(c) We also have
$$
\lim_{t\to \infty }\intot \Big(\frac{\Upsilon(t-s)}{m_t} - \frac{\alpha_0}{\mu}e^{-\alpha_0 s} \Big)^2 m'_s ds = 0.
$$

\vip

(d) Consider a real sequence $(\eta_n)_{n\geq 1}$ such that $\lim_{n\to \infty} \eta_n=0$. Then
we have the property that 
$\lim_{t\to\infty} e^{-\alpha_0 t}\sum_{n\geq 1} \eta_n \intot \varphi^{\star n}(t-s) e^{\alpha_0 s/2}ds= 0$.

\vip

(e) For any pair of locally bounded functions $u,h:(0,\infty)\mapsto \rr$ such that
$u=h+u\star\varphi$, there holds $u=h+h\star \Gamma$.
\end{lem}

\begin{proof} We easily deduce from our assumptions on $\varphi$ that there is some constants $C>0$, $p>0$
such that for all $t\geq 0$, $\varphi(t) \leq C(1+t)^p$ (in particular, $\varphi$ is locally bounded). 
Hence its Laplace transform
is clearly well-defined on $(0,\infty)$, of class $C^\infty$, and
$\lim_{\alpha\to \infty} \cL_\varphi(\alpha)=0$.
Furthermore, $\cL_\varphi(0)=\int_0^\infty \varphi(t)dt\in(1,\infty]$. 
Hence, there indeed exists a unique $\alpha_0>0$ such that $\cL_\varphi(\alpha_0)=1$.
We now divide the proof into several steps.

\vip

{\it Step 1.} We first prove that $\Gamma$ is locally bounded. To this end, 
we introduce $\Gamma_n=\sum_{k=1}^n \varphi^{\star k}(t)$ and observe that $\Gamma_{n+1}=\varphi+\Gamma_n\star\varphi$.
Since $\varphi$ is locally bounded, Lemma \ref{grrrr}-(iii) allows us to conclude that $\sup_n \Gamma_n$ is 
locally bounded, whence the result.

\vip

{\it Step 2.} Here we prove (e). Since $h$ is locally bounded and since $\varphi$ is locally integrable,
we easily deduce from Lemma \ref{grrrr}-(i) that the equation $v=h+v\star\varphi$ (with unknown $v$)
has at most one locally bounded solution. Since both $u$ and $h+h\star \Gamma$ are locally bounded solutions,
the conclusion follows.

\vip

{\it Step 3.} The aim of this step is to verify that 
$\Gamma(t) \sim (a_0\alpha_0^2/\mu) e^{\alpha_0 t}$ as $t\to \infty$.
Observe that $\Gamma$ solves $\Gamma=\varphi + \Gamma \star \varphi$ and introduce $u(t)=\Gamma(t)e^{-\alpha_0t}$
and $f(t)=\varphi(t)e^{-\alpha_0t}$. One easily checks that $u=f+u\star f$. We now apply Theorem 4 of Feller \cite{f}.
We have $\int_0^\infty f(t)dt=1$ by definition of $\alpha_0$. We set $b_1=\int_0^\infty t f(t)dt
= \int_0^\infty t\varphi(t)e^{-\alpha_0 t}dt$ and $b_2=\int_0^\infty t^2 f(t)dt
= \int_0^\infty t^2\varphi(t)e^{-\alpha_0 t}dt$,
which clearly both converge, since $\varphi(t) \leq C(1+t)^p$.
Finally, it is not difficult to check that $f(t)$, $tf(t)$ and $t^2f(t)$ have a bounded total variation on 
$[0,\infty)$ since we have assumed that $t\mapsto \intot|d\varphi(s)|$ has at most polynomial growth. 
Thus Feller \cite[Theorem 4]{f}
tells us that $u(t) \to 1/b_1$ as $t\to \infty$, which gives $\Gamma(t) \sim e^{\alpha_0 t}/b_1$.
This ends the proof, since $1/b_1=a_0\alpha_0^2/\mu$ by definition of $a_0$.

\vip

{\it Step 4.} We now conclude the proof of (a) and (b). 
Recall that $m_t=\mu t + \intot \int_0^s \varphi(s-u)m'_u du ds$,
so that $m'=\mu+ \varphi\star m'$. Applying (e), we deduce that $m'=\mu+\mu\star\Gamma=\mu(1+\Upsilon)$.
By Step 3, we know that $\Gamma(t) \sim (a_0\alpha_0^2/\mu)e^{\alpha_0 t}$ as $t\to \infty$.
This obviously implies that $\Upsilon(t) \sim (a_0\alpha_0/\mu)e^{\alpha_0 t}$, whence
$m'_t \sim a_0\alpha_0 e^{\alpha_0 t}$ and finally $m_t=\intot m'_s ds \sim a_0 e^{\alpha_0 t}$.
Finally, (b) directly follows from (a) and the facts that $\Gamma$, $\Upsilon$, $m$ and $m'$
are nonnegative and locally bounded.

\vip

{\it Step 5.} Point (d) is not very difficult:
using that $\varphi(t)\leq C (1+t)^p$, we see that 
$\varphi^{\star n}(t)\leq C_n(1+t)^{p_n}$ for some constants $C_n>0$ and $p_n>0$.
We next introduce $\e_k=\sup_{n\geq k} |\eta_n|$, which decreases to $0$ as $k\to\infty$,
and write, for any $k\geq 1$,
\begin{align*}
\limsup_{t\to\infty} e^{-\alpha_0 t}\sum_{n\geq1} |\eta_n| \intot \varphi^{\star n}(t-s) e^{\alpha_0 s/2}ds
\leq& \limsup_{t\to\infty} e^{-\alpha_0 t}\sum_{n=1}^k |\eta_n| \intot C_n (t-s)^{p_n} e^{\alpha_0 s/2}ds\\
&+ \e_k \limsup_{t\to\infty} e^{-\alpha_0 t} \sum_{n=k+1}^\infty \intot \varphi^{\star n}(t-s) e^{\alpha_0 s/2}ds.
\end{align*}
The first term on the RHS is of course $0$ (for any fixed $k$). We can bound the second one, using
(b), by
$$
\e_k \limsup_{t\to\infty} e^{-\alpha_0 t} \intot \Gamma(t-s)e^{\alpha_0 s/2}ds \leq  C \e_k 
\limsup_{t\to\infty}e^{-\alpha_0 t} \intot e^{\alpha_0(t-s)}
e^{\alpha_0 s/2}ds \leq C\e_k. 
$$
Letting $k$ tend to infinity concludes the proof.

\vip

{\it Step 6.} It only remains to check point (c). We use the Lebesgue dominated convergence theorem.
Define $h^t_s=(\Upsilon(t-s)/m_t - \alpha_0 e^{-\alpha_0 s}/\mu)^2 m'_s \indiq_{\{s\leq t\}}$.
We have to prove that $\lim_{t\to\infty} \int_0^\infty h^t_s ds =0$. First, it is obvious
from (a) that for $s>0$ fixed, $\lim_{t\to\infty} h^t_s =0$.
Next, we use (b) to write (for $t$ large enough so that $m_t\geq ce^{\alpha_0t}$) 
$h^t_s \leq C (e^{-\alpha_0 s})^2 e^{\alpha_0 s} \leq C e^{-\alpha_0 s}$, which does not depend on $t$ and is integrable
on $(0,\infty)$.
\end{proof}

We next consider briefly a {\it vector} convolution equation.

\begin{lem}\label{vectconv}
Consider a family $\bmu=(\mu_i)_{i\in\zz^d}$ of real numbers such that $0\leq \mu_i\leq C$ for all $i\in\zz^d$,
the stochastic matrix $(A(i,j))_{i,j\in\zz^d}$ defined by \eqref{dfA}, and a locally integrable function
$\varphi: (0,\infty)\mapsto [0,\infty)$.
The equation
$\bm_t=\bmu t + \intot \varphi(t-s) A \bm_s ds$ with unknown $\bm=(m^i_t)_{t\geq 0, i\in\zz^d}$
has a unique solution such that for all $t\geq 0$, $\sum_{i\in\zz^d} 2^{-|i|}\sup_{[0,t]}m^i_s <\infty$.
Furthermore, $m^i$ is of class $C^1$ on $[0,\infty)$ for all $i\in\zz^d$, and it holds that
$\bm_t'=\bmu + \intot \varphi(t-s) A \bm'_s ds$ and 
$\bm_t'= \Big(I + \sum_{n\geq 1} A^n \intot \varphi^{\star n}(s)ds \Big) \bmu$.
Finally, $u_t=\sup_{i\in\zz^d} \sup_{[0,t]}(m^i_s)'$ is finite for all $t\geq 0$
and it holds that $u_t \leq C + \intot \varphi(t-s) u_s ds$.
\end{lem}

\begin{proof} We proceed in a few steps.

\vip

{\it Step 1.}
We first note that for any vector $\bx=(x_i)_{i\in\zz^d}$, it holds that 
$\sum_{i \in \zz^d} 2^{-|i|}|(A\bx)_i| \leq 2 \sum_{i \in \zz^d} 2^{-|i|}|x_i|$.
This easily follows from the explicit form of $A$.

\vip

{\it Step 2.} We next check uniqueness. Consider two solutions $\bm$ and $\btm$ satisfying the required condition 
and put $h_t=\sum_{i\in\zz^d} 2^{-|i|}\sup_{[0,t]}|m^i_t-\tm^i_t|$. 
We have $h_t \leq \sum_{i\in\zz^d} 2^{-|i|} \intot \varphi(t-s)|(A (\bm_s-\btm_s)_i|ds$.
Using Step 1, we deduce that $h_t \leq 2 \intot \varphi(t-s) h_s ds$
and thus $h_t=0$ by Lemma \ref{grrrr}-(i).

\vip

{\it Step 3.} We define $\bm_t'= \Big(I + \sum_{n\geq 1} A^n \intot \varphi^{\star n}(s)ds \Big) \bmu$.
Using that $A$ is stochastic and that $\bmu$ is bounded (by $C$), we easily deduce that for all $i\in\zz^d$,
$(m^i_t)'\leq C (1+\sum_{n\geq 1} \intot \varphi^{\star n}(s)ds)$. This function is
locally bounded because $\varphi$ is locally integrable: use that 
$\Upsilon_k(t)=\sum_{n= 1}^k \intot \varphi^{\star n}(s)ds$ satisfies $\Upsilon_k(t) \leq \intot \varphi(s)ds
+\intot \varphi(t-s) \Upsilon_k(s)ds$ and use Lemma \ref{grrrr}-(iii), which provides a uniform (in $k$) bound.
Consequently, $u_t=\sup_{i\in\zz^d} \sup_{[0,t]}(m^i_s)'$ is finite for all $t\geq 0$.
Similar arguments show that $(m^i_t)'$ is continuous on $[0,\infty)$,
because $|(m^i_{t+h})'-(m^i_t)| \leq C \sum_{n\geq 1} \int_t^{t+h} \varphi^{\star n}(s)ds$.

\vip

{\it Step 4.} A straightforward consequence of the definition of $\bm_t'$ is that it solves
$\bm_t'=\bmu + \intot \varphi(t-s) A \bm'_s ds$. Using that $A$ is stochastic and that $\bmu$ is bounded,
we immediately deduce that $u_t \leq C + \intot \varphi(t-s) u_s ds$.
We finally define, for each $i\in\zz^d$, $m^i_t=\intot (m^i_s)'ds$.
Integrating the equation satisfied by $\bm'$ and using Lemma \ref{tlt}, we find that
$\bm=(m^i_t)_{i\in\zz^d,t\geq 0}$ is indeed a solution to $\bm_t=\bmu t + \intot \varphi(t-s) A \bm_s ds$.
It only remains to check that $\sum_{i\in\zz^d}2^{-|i|}\sup_{[0,t]} m^i_t <\infty$ for all $t\geq 0$,
but this obviously follows from the facts that $u_t$ is locally bounded and that $\bm_0=0$.
\end{proof}

\end{document}